\documentclass[10pt]{amsart}

\usepackage{amsmath}
\usepackage{amsthm}
\usepackage{hyperref}
\usepackage{ragged2e}
\usepackage{enumitem}
\usepackage{tikz-cd}
\usepackage{mathtools}

\usepackage[
backend=biber,
style=alphabetic,
]{biblatex}

\usepackage{amssymb}
\newcommand{\bb}{\mathbb}

\newcommand{\vphi}{\varphi}

\newcommand{\Res}{\operatorname{Res}}
\newcommand{\res}{\operatorname{res}}

\newcommand{\fra}{\mathfrak}
\newcommand{\ovl}{\overline}

\newcommand{\berkP}{\mathbb{P}^{1,an}}
\newcommand{\berkA}{\mathbb{A}^{1,an}}

\newcommand{\cal}{\mathcal}
\newcommand{\Spec}{\operatorname{Spec}}

\newcommand{\eps}{\varepsilon}

\newcommand{\h}{\widehat{h}}

\newcommand{\Rat}{\operatorname{Rat}}

\newcommand{\GL}{\operatorname{GL}}
\newcommand{\rat}{\operatorname{rat}}
\newcommand{\SL}{\operatorname{SL}}
\newcommand{\Ht}{\widehat{H}}
\newcommand{\Gauss}{\operatorname{Gauss}}
	
\newtheorem{theorem}{Theorem}
\numberwithin{theorem}{section}

\newtheorem{proposition}[theorem]{Proposition}

\numberwithin{example}{section}

\newtheorem{definition}{Definition}
\numberwithin{definition}{section}

\makeindex
\title{On the Number of Small Points for Rational Maps}
\author{Jit Wu Yap}
\begin{document}

\begin{abstract}
Let $K$ be a number field and $f: \bb{P}^1 \to \bb{P}^1$ a rational map of degree $d \geq 2$ with at most $s$ places of bad reduction, where we include all archimedean places. We prove that there exists constants $c_1,c_2 > 0$, depending only on $d$ and not on $f$ or $K$, such that 
$$
\# \left\{ x \in \bb{P}^1(K) \mid \h_f(x) \leq \frac{c_1}{s} h_{\rat_d}(\langle f \rangle) \right\} \leq c_2 s \log(s).
$$
Here, $\rat_d$ is the moduli space of rational maps up to conjugacy, $h_{\rat_d}$ is an ample height and $\langle f \rangle$ is the equivalence class associated to $f$. This gives a uniform version of Baker \cite[Theorem 1.14]{Bak06} as well as generalizing the results of Benedetto \cite{Ben07} and Looper \cite{Loo19} from polynomials to rational maps. The main tool used is the degeneration of sequences of rational maps by Luo \cite{Luo21, Luo22} which has been recently formalized by Favre--Gong \cite{FG24} via Berkovich spaces.
\end{abstract}

\maketitle

\section{Introduction}
Let $K$ be a number field and let $f: \bb{P}_K^1 \to \bb{P}_K^1$ be a rational map of degree $d \geq 2$. Let $\h_f$ denote the canonical height of $f$. Then for any finite extension $L/K$ of degree $D$, Baker \cite{Bak06} proves that there exists constants $A,B > 0$, depending on $f$, such that 
\begin{equation} \label{eq: IntroBaker1}
\# \left\{ x \in \bb{P}^1(L) \mid \h_f(x) \leq  \frac{A}{D} \right\} \leq BD \log D.
\end{equation}
For any positive integer $s \geq 1$ and $d \geq 2$, we prove a uniform version of \eqref{eq: IntroBaker1} when $f$ is of degree $d$ and has at most $s$ places of bad reduction. We let $\Rat_d$ be the moduli space of degree $d$ rational maps and $\rat_d$ be the moduli space of degree $d$ rational maps up to conjugation by $\GL_2$. Given $f \in \Rat_d(K)$ for a field $K$, we will let $\langle f \rangle$ denote its conjugacy class in $\rat_d(K)$. Let $h_{\rat_d}$ be a height function for $\rat_d$ coming from an ample line bundle that we shall fix once and for all. Given a place $v \in M_K$, we say that $f$ has bad reduction over $v$ if $v$ is archimedean, or $v$ is non-archimedean and no conjugate of $f$ extends to an endomorphism over $\bb{P}^1_{O_{K_v}}$, where $O_{K_v} = \{x \in K \mid |x|_v \leq 1\}$ is the valuation ring.

\begin{theorem} \label{IntroUniformBaker1}
Let $K$ be a number field and $f: \bb{P}_K^1 \to \bb{P}_K^1$ a rational map of degree $d \geq 2$ with $s$ places of bad reduction, where we include all archimedean places. There exists constants $c_1,c_2 > 0$, depending only on $d$ and independent of $f$ and $K$, such that 
$$
\# \left\{ x \in \bb{P}^1(K) \mid \h_f(x) \leq \frac{c_1}{s} h_{\rat_d}(\langle f \rangle) \right\} \leq c_2 s \log(s).
$$
\end{theorem}

Given any extension $L/K$ where $[L:K] = D$, we may view $f$ as a rational map over $L$ with at most $Ds$ many places of bad reduction. Then Theorem \ref{IntroUniformBaker1} implies Baker's result \eqref{eq: IntroBaker1} and gives a uniform version of it. When $K$ is a function field, we have analogous results (Theorem \ref{UniformBakerTheorem2}). Although we have an explicit dependence on the number of places of bad reduction, our constants $c_1$ and $c_2$ are ineffective as we shall see from our methods. Taking $\h_f(x) = 0$, this gives an upper bound on the number of $K$-preperiodic points that generalizes Benedetto's result \cite{Ben07} for polynomials. Note that Troncoso \cite{Tro16} had previously proved a uniform bound on $K$-preperiodic points that only depends only the degree $d$ and the number of places of bad reduction using a $S$-integrality approach. However his bound depends exponentially on $s$ whereas ours is a $O(s \log s)$ dependence which matches that of Benedetto's. 
\par 
A simple corollary of Theorem \ref{IntroUniformBaker1} is that we obtain a uniform lower bound on the canonical height of a non-preperiodic point. 

\begin{theorem} \label{IntroUniformBaker2}
Let $K$ be a number field and $f: \bb{P}_K^1 \to \bb{P}_K^1$ a rational map of degree $d \geq 2$ with $s$ places of bad reduction. Then there exists a constant $c_3 > 0$, depending only on $d$ and independent of $f$ and $K$, such that for $x \in \bb{P}^1(K)$, we have 
$$\h_f(x) = 0 \text{ or } \h_f(x) \geq \frac{1}{d^{c_3 s \log s}} h_{\rat_d}(\langle f \rangle).$$
\end{theorem}

This answers a weaker version of a conjecture of Silverman \cite[Conjecture 4.98]{Sil07} for rational maps with $s$ places of bad reduction and generalizes Looper's result for polynomials \cite{Loo19} to rational maps. It also implies a result of Silverman \cite{Sil81} for elliptic curves as Theorem \ref{IntroUniformBaker2} applies to Lattès maps. Compared to Looper's result, we also obtain a better dependence on $s$ although our constant $c_3$ is ineffective unlike Looper's. We note that over function fields, the question of bounding the number of preperiodic points uniformly over families, independent of the number of places of bad reduction, has been answered in some special cases, see \cite{HS88, DP20, Loo21, DF24}. 
\par 
Let us now explain the main ideas behind the proof. The general approach is the same as Baker's. For each place $v$ of $M_K$, we consider the dynamical Arakelov--Green function $g_{f,v}(x,y)$ defined as follows. We fix a lift $F: \bb{A}^2 \to \bb{A}^2$ of $f$ over $K$. Then if $\tilde{x},\tilde{y}$ denote lifts of $x,y \in \bb{P}^1(K)$ to $\bb{A}^2$, we have
$$g_{f,v}(x,y) = -\log |\tilde{x} \wedge \tilde{y}| + \Ht_F(\tilde{x}) + \Ht_F(\tilde{y}) - \frac{1}{d(d-1)} \log |\Res(F)|.$$
where $\tilde{x} \wedge \tilde{y} = x_0y_1 - x_1y_0$ and $\Res$ denotes the homogeneous resultant. Baker then shows that given $z_1,\ldots,z_N \in \bb{P}^1(K)$, we have 
\begin{equation} \label{eq: IntroBakerLower1}
\sum_{i \not = j} g_{f,v}(z_i,z_j) \geq -O_{f}(N \log N)
\end{equation} 
and it is not hard to make bounds explicit in $f$. He then looks at an archimedean place $w$ and concludes the existence of an open cover $U_1,\ldots,U_n$ and a $\delta > 0$ such that 
\begin{equation} \label{eq: IntroBakerLower2}
g_{f,w}(x,y) > \delta \text{ for } x,y \in U_i.
\end{equation}
Now given $N$ points of small height, we may choose $M = \frac{N}{n}$ of them, say $z_1,\ldots,z_M$, to live in the same open set $U_j$. Combining \eqref{eq: IntroBakerLower1} and \eqref{eq: IntroBakerLower2} and summing up over all $v \in M_K$, we obtain a contradiction if $N$ is too large as 
$$\sum_{v \in M_K} \sum_{i \not = j} N_v g_{f,v}(z_i,z_j) = 2(M-1) \sum_{i=1}^{M} \h_f(z_i).$$
To obtain Theorem \ref{IntroUniformBaker1}, we have to upgrade \eqref{eq: IntroBakerLower2} to be uniform over rational maps with at most $s$ places of bad reduction, and have the constant scale as $\h_{\rat_d}(\langle f \rangle)$ grows. The next theorem gives us this uniformity. Let
$$|\res(f)| = \sup_{\vphi \in \GL_2(K)} |\Res(\vphi \circ f \circ \vphi^{-1})|$$
denote the minimal resultant of $f$. 

\begin{theorem} \label{IntroUniformBaker3}
There exists a constant $\delta > 0$ and a positive integer $N$, depending only on $d$, such that for any degree $d$ rational map $f$ over a complete algebraically closed field $(K, | \cdot |)$, we may cover $\bb{P}^1(K)$ with $N$ open sets $U_1,\ldots,U_N$ such that 
$$g_{f}(x,y) > \delta (-\log |\res(f)|).$$ 
for any $x,y \in U_i$. If $K = (\bb{C}, | \cdot |)$ for the standard norm, we may further assume that 
$$g_f(x,y) > \delta \max\{1, -\log |\res(f)|).$$
\end{theorem}

Using Theorem \ref{IntroUniformBaker3}, we can prove Corollary \ref{IntroUniformBaker2} with $h_{\rat_d}(\langle f \rangle)$ replaced with a minimal resultant ``height" $h_{\res}(f)$, which we define by 
$$h_{\res}(f) = \sum_{v \in M_K} N_v \log|\res(f)|_v.$$
But it turns out that $h_{\res}(f)$ is comparable with $h_{\rat_d}(\langle f \rangle)$ and thus we can conclude Theorem \ref{IntroUniformBaker1}.
\par 
To prove Theorem \ref{IntroUniformBaker3}, we apply the idea of degeneration of sequences of rational maps due to Luo \cite{Luo21, Luo22} and formalized by Favre--Gong \cite{FG24} using the theory of Berkovich spaces over Banach rings. Their method a priori applies only to sequences of rational maps over $\bb{C}$ but it is straightforward to modify for it to work over not just a fixed non-archimedean field $K$, but even over a sequence of non-archimedean fields $(K_n)$. By an argument similar to sequential compactness, Theorem \ref{IntroUniformBaker3} then reduces to verifying \eqref{eq: IntroBakerLower2} for $f$ over an arbitrary complete algebraically closed non-archimedean field $K$ where $f$ has bad reduction. This then follows from the results in Baker \cite{Bak09}.
\par 
The idea of using non-archimedean fields to understand degenerations of meromorphic families of rational maps $(f_t)_{t \in \bb{D}^{*}}$ over the puctured unit disc $\bb{D}^{*}$ goes back to Kiwi \cite{Kiw06, Kiw14}. DeMarco--Faber \cite{DF14, DF16} used such degenerations to understand the behavior of the equilibrium measure $\mu_{f_t}$ as $t \to 0$. This was also done independently by Favre \cite{Fav20}, using a hybrid space, who extended DeMarco--Faber's work to the higher dimensional setting of $\bb{P}^n$. In DeMarco--Krieger--Ye \cite{DKY20}, meromorphic degenerations were used prove a uniform bound on the number of common preperiodic points for Lattes maps. 
\par 
Later on, Lemanissier and Poineau \cite{LP24} developed the theory of Berkovich spaces over $\Spec \bb{Z}$. Poineau \cite{Poi24a, Poi24b} then builds a good notion of potential theory for the projective line over such spaces with application to dynamics. Poineau further shows that it is possible to degenerate the situation over $\bb{C}_p$, to the setting of a non-archimedean field with residue characteristic $0$, even as the prime $p$ changes which leads to a proof of Bogomolov--Fu--Tschinkel. The idea that it should be possible to consider degenerations over an arbitrary sequence of non-archimedean fields is directly inspired by Poineau's work. 
\par 
Let us explain in more detail how our degenerations work. A meromorphic family of rational maps $\{f_t\}_{t \in \bb{D}^{*}}$ over a punctured unit disc $\bb{D}^{*}$ can be viewed as a rational map $f(z)$ over the field $K = \bb{C}((t))$ of Laurent series. If we let $\bb{L}$ be the field of Puiseux series, which is the completion of the algebraic closure of $K$, then the dynamics of $f_t$ approaches the dynamics of $f(z)$ over $\bb{L}$ according to the rescaling factor $\log |t|^{-1}$. For example in Favre \cite{Fav20}, it was shown that the rescaled Lyapunov exponent $\frac{\chi_{f_t}}{\log |t|^{-1}}$ converges to the non-archimedean Lyapunov exponent $\chi_{f(z)}$. 
\par 
In general given a family of rational maps $B$, to understand the behavior of our rational maps as it approaches the boundary of $B$ we would like to construct a compactification $\ovl{B}$ where the additional points in $\ovl{B} \setminus B$ have meaningful interpretation. In fact, Faltings \cite{Fal21}, Song \cite{Son23} and Yuan \cite{Yua24} managed to use the compactification of the moduli space of curves to understand the behavior of the Arakelov--Green function of a smooth projective curve over $\bb{C}$ as the curve approaches the boundary. 
\par 
However for rational maps, there has been no satisfying compactification as of yet although there have been a number of proposals \cite{Sil98, Dem07}. For polynomials it seems that the theory of trees by DeMarco--McMullen \cite{DeM08} would give a compactification that is suitable for our purposes over $\bb{C}$ and indeed, the theory of trees has been used by Looper \cite{Loo19, Loo21, Loo21b} to obtain results for polynomials in the flavor of our main theorems. 
\par 
Luo \cite{Luo21, Luo22} gets around this problem by using ultrafilters to instead construct limits for any sequence of rational maps $(f_n)$ of fixed degree over $\bb{C}$ according to the rescaling parameter $\log |\res(f_n)|^{-1}$. If our sequence degenerates to the boundary, then these limits exists as a rational map $f_{\omega}$ over an extremely large non-archimedean field $\cal{H}(\omega)$. Although this does not give a compactification of our moduli space, it is good enough to prove properties about rational maps near the boundary if we can check that our property holds for rational maps over $\cal{H}(\omega)$ and that the property is an ``open" condition. 
\par 
Favre--Gong \cite{FG24} formalizes Luo's construction by using the Berkoivich space over a suitable Banach ring. It is then straightforward to modify their arguments to do degenerations of rational maps $(f_n)$ over an arbitrary sequence of non-archimedean fields $(K_n)$. In particular, for an artbirary sequence of primes $(p_n)$, we may take $K_n = \bb{C}_{p_n}$ which allows us to prove uniform results over all primes $p$.
\par 
We now describe the proof of Theorem \ref{IntroUniformBaker3}, at least in the archimedean case. The existence of these $N$ open sets $U_i$ is invariant under conjugation and so we may assume that $|\Res(f)| = |\res(f)|$. We now assume our Theorem is false and hence we obtain for each $n$, a rational map $f_n$ of degree $d \geq 2$ such that there do not exists an open cover of $n$ open sets $U_i$ such that 
$$g_f(x,y) \geq \frac{1}{n} \max\{1,\log|\res(f_n)|^{-1}\} \text{ for all } x,y \in U_i.$$
For simplicity, we assume that $\log|\res(f_n)| \to 0$. Then given a non-principal ultrafilter $\omega$, Favre--Gong's degeneration produces for us a rational map $f_{\omega}$ of degree $d \geq 2$ over $\cal{H}(\omega)$ with bad reduction. But Baker \cite{Bak09} proved that over an arbitrary algebraically closed non-archimedean field, there exists $N$ open sets $U_{i,\omega}$ covering $\berkP_{\cal{H}(\omega)}$, such that 
$$g_{f_{\omega}}(x,y) > \delta \text{ for all } x,y \in U_i.$$
We now use continuity arguments to spread these open sets such that for infinitely many $n$, we have $N$ open sets $U_{i,n}$ covering $\bb{P}^1(\bb{C})$ such that
$$g_{f_n}(x,y) > \delta \log|\res(f_n)|^{-1},$$
where we add in a $\log |\res(f_n)|^{-1}$ factor as that is our rescaling factor. This gives a contradiction and we are done. 

\subsection{Acknowledgements} The author would like to thank Niven Achenjang, Benjamin Church, Laura DeMarco, Charles Favre, Chen Gong, Nicole Looper, Niki Myrto Mavraki and Junyi Xie for helpful discussions regarding the paper. 

\section{Heights and Resultants}
\subsection{Heights over Global Fields}
Let $K$ be either a number field or the function field of a curve over some finite field $\bb{F}_q$. Then if $M_K$ denotes the set of places of $K$, we have the product formula 
$$\prod_{v \in M_K} |x|_v = 1$$
for all $x \in K^{\times}$. If $K$ is a number field, we can define the absolute Weil height on $\bb{P}^{1}(K)$ by 
$$h([x_0:x_1]) = \frac{1}{[K:\bb{Q}]} \sum_{v \in M_K} \log \max\{|x_0|_v, |x_1|_v\}.$$
The product formula ensures that this is independent of our choice of $[x_0:x_1]$ and the normalization ensures that the definition remains invariant when passing to finite extensions $L$. In particular we obtain a height function on $\bb{P}^1(\ovl{K})$. 
\par For function fields $K$, we do not have a universal base field like $\bb{Q}$ and instead we take $K$ to be our base field. Then for any finite extension $L/K$, we define for $x \in \bb{P}^1(L)$
$$h([x_0:x_1]) = \frac{1}{[L:K]} \sum_{v \in M_L} \log \max\{|x_0|_v, |x_1|_v\}.$$
This is then independent of the extension $L$ and so gives us a well-defined height on $\bb{P}^1(\ovl{K})$. Now let $f: \bb{P}^1 \to \bb{P}^1$ be a rational map defined over $K$ with degree $d \geq 2$. Following Call--Silverman \cite{CS93}, we can define a canonical height 
$$\h_f(x) = \lim_{n \to \infty} \frac{1}{d^n} h(f^n(x))$$
for $x \in \bb{P}^1(\ovl{K})$. The canonical height satisfies $|\h_f(x) - h(x)| \leq O(1)$ and $\h_f(f(x)) = d \h_f(x)$. If $x$ is a preperiodic point, i.e. satisfies $f^m(x) = f^n(x)$ for some $m > n$, then the canonical height $\h_f(x)$ is exactly $0$. When $K$ is a number field the converse holds due to the Northcott property but it is not necessarily true if $K$ is a function field. However if $f$ is non-isotrivial, the converse is still true by Baker \cite{Bak09}.

\subsection{Moduli Space of Rational Maps and Resultants} Now let $K$ be an arbitrary field. Given a rational map $f$ of degree $d \geq 2$ over $K$, we define a homogeneous lift $F = (P(x,y), Q(x,y))$ to be a pair of degree $d$ homogeneous polynomials in $x,y$ such that 
$$f([x_0:x_1]) = [P(x_0:x_1):Q(x_0:x_1)].$$
Working over $\bb{A}^1$, this is equivalent to saying $f(z) = \frac{F(z,1)}{G(z,1)}$. Conversely any such pair of polynomials $P(x,y),Q(x,y)$ corresponds to a rational map as long as they have no common factor. This can be ensured by requiring that the homogeneous resultant $\Res(F) = \Res(P,Q)$ is non-zero. The homogeneous resultant $\Res(F,G)$ is a homogeneous integer polynomial in the coefficients of $F$ and $G$ and this allows us to define the moduli space of degree $d$ rational maps $\Rat_d$ as $\bb{P}^{2d+1} \setminus \{\Res(F) = 0\}$, where we view the coordinates of $\bb{P}^{2d+1}$ as $[a_0: \cdots : a_d: b_0 : \cdots : b_d]$ and we let a point correspond to the polynomials
$$P(x,y) = \sum_{i=0}^{d} a_i x^i y^{d-i}, Q(x,y) = \sum_{j=0}^{d} b_j x^j y^{d-j}.$$
Observe that as $\Res(F,G)$ is defined over $\bb{Z}$, our moduli space $\Rat_d$ is defined over $\bb{Z}$ and is in fact a fine moduli space for degree $d$ rational maps \cite{Sil98}. For different lifts $F$ of $f$, the resultant of $F$ will vary but if $K$ comes with an absolute value $|\cdot |$, it is possible to define $|\Res(f)|$ by 
$$|\Res(f)| = \frac{|\Res(F)|}{\max\{|a_i|^{2d+2}, |b_j|^{2d+2}\}}$$
where $F = (P,Q)$ is a lift of $f$ and $a_i$ and $b_j$ are the coefficients of $P,Q$ respectively. Then due to the normalization, $|\Res(f)|$ is independent of the choice of the lift $F$. 
\par 
Now there is a $\SL_2$-action on $\Rat_d$ as follows. An element of $\GL_2$ gives us an automorphism of $\bb{P}^1$, which we can denote by $\vphi$. Then we let $\vphi$ act on $\Rat_d$ by setting $(\vphi,f) \mapsto \vphi \circ f \circ \vphi^{-1}$. We can then consider rational maps up to such conjugacy. Silverman \cite{Sil98} shows that there is a coarse moduli space $\rat_d$ for this functor and that $\rat_d$ is an affine scheme over $\bb{Z}$ such that over an algebraically closed field $K$, we have $\rat_d(K) = \Rat_d(K)/\SL_2(K)$. For $d = 2$, we have an explicit isomorphism $\rat_2 \simeq \bb{A}^2$ \cite{Mil93}. We will mention that the analogous spaces for higher dimensions, i.e. endomorphisms of $\bb{P}^n$, have also been studied in detail \cite{PST09, Lev11}.  
\par 
When $K$ is a complete algebraically closed field under $| \cdot |$, we can then define the minimal resultant, which we will denote by $|\res(f)|$ by setting
$$|\res(f)| = \sup_{\vphi \in \SL_2(K)} |\Res(\vphi \circ f \circ \vphi^{-1})|.$$
For non-archimedean $K$, the minimal resultant is important because it detects bad reduction. In this case, we always have $|\res(f)| \leq 1$ and $|\res(f)| < 1$ occurs exactly when no conjugate of $f$ extends to a rational map over $\Spec O_K$ where $O_K = \{x \in K \mid |x| \leq 1\}$. This is equivalent to saying that $f$ has bad reduction over $K$. Thus the minimal resultant $|\res(f)|$ plays the role of the $j$-invariant of an elliptic curve $E/K$ and so it has been studied in great detail \cite{STW14, Rum15}. 
\par 
Since $|\res(f)|$ is clearly invariant under conjugation, the function $|\res( \cdot )|$ descends to a function on $\rat_d(K)$ too.
\par 
We now want to prove that $-\log |\res(f)|$ behaves like a ``local height" for $\rat_d$. We explain what we mean as follows. Since $\rat_d$ is an affine scheme over $\Spec \bb{Z}$, we may find an embedding $\rat_d \xhookrightarrow{} \bb{A}^N$ for some $N \geq 1$. This gives us coordinates $(a_1,\ldots,a_N)$ for the points of $\rat_d(K)$ and when $K$ has an absolute value $| \cdot |$, we can define a local height $\log^+|\langle f \rangle|$ for $\langle f \rangle \in \rat_d(K)$ by setting
$$\log^+ |\langle f \rangle| = \max_{1 \leq i \leq n} \log^+|a_i|.$$
For a global field $K$ with an absolute value coming from a place $v \in M_K$, this is exactly the local height coming from using the embedding $\bb{A}^N \xhookrightarrow{} \bb{P}^N$ along with the hyperplane divisor $D = \{x_0 = 0\}$. Our aim is to prove that $-\log |\res(f)|$ is, up to some constants, at least as big as $\log^+|\langle f \rangle|$. 

\begin{proposition} \label{LocalResComparison1}
Let $(K, |\cdot |)$ be a complete valued field. Then there exists constants $A,B > 0$, depending only on $d$ and independent of $K$, such that for any $f \in \Rat_d(K)$,
$$A( -\log|\res(f)|) + B \geq \log^+|\langle f \rangle|.$$
If $| \cdot |$ is non-archimedean, then we may take $B = 0$. 
\end{proposition}

\begin{proof}
We first give affine charts for $\Rat_d$ over $\Spec \bb{Z}$. Let $x_0,\ldots,x_{2d+1}$ be the coordinates for $\bb{P}^{2d+1}$ and let $U_i = \bb{P}^n \setminus V(x_i)$, which is isomorphic to $\bb{A}^{2d+1}$. If $f$ is represented by $[x_0: \cdots : x_{2d+1}]$ with $x_i \not = 0$, we let $F_i$ be the lift of $f$ after setting $x_i = 1$. Then $\Rat_d \cap U_i$ is given by $\bb{A}^{2d+1} \setminus V(\Res(F_i))$ which is isomorphic to $V_i = \bb{A}^{2d+2}/(y \Res(F_i) - 1)$ by the map $(x_1,\ldots,x_{2d+1}) \mapsto (x_1,\ldots,x_{2d+1}, \Res(F_i)^{-1})$.
\par 
We now have a morphism of affine schemes $\pi: \Rat_d \to \rat_d$ over $\bb{Z}$. Given any point $f \in \Rat_d(K)$, we let $i$ be an index such that $|x_i|$ is maximal and write $f$ in the affine chart $V_i$. Then in this chart as $\pi: V_i \to \rat_d$ is a morphism over $\bb{Z}$, viewing $\rat_d$ as a closed subscheme of $\bb{A}^N$, our morphism is of the form
$$\left(\frac{x_0}{x_i}, \cdots, \frac{x_{2d+1}}{x_i}, \Res(F_i)^{-1} \right) \mapsto (P_1(\ovl{x}),\ldots,P_N(\ovl{x})) $$
where each $P_i$ is a polynomial with $\bb{Z}$-coefficients with $2d+2$ variables and is being evaluated on 
$$\ovl{x} = \left( \frac{x_0}{x_i},\ldots,\frac{x_{2d+1}}{x_i}, \Res(F_i)^{-1} \right).$$ 
Now by the choice of our chart, we have $|\frac{x_j}{x_i}| = 1$ and furthermore, we have $\Res(F_i) = \Res(f)$ as the maximum valuation of the homogeneous coordinates is exactly $1$. It then follows immediately from the triangle inequality that 
$$\log^+ \left|P_j \left(\frac{x_0}{x_1},\ldots,\frac{x_{2d+1}}{x_i},\Res(F_i)^{-1} \right) \right| \leq A(-\log \Res(f)) + B$$
for some constants $A,B > 0$ that depend only on $P_j$ and not on $K$. Furthermore if $| \cdot |$, the strong triangle inequality implies that we may take $B = 0$ as desired. Finally as this holds for any conjugate of $f$, we may replace $|\Res(f)|$ with $|\res(f)|$ as desired. 
\end{proof}

For a global field $K$, we may define a minimal resultant ``height" on $\Rat_d(\ovl{K})$ by the formula 
$$h_{\res}(f) = \frac{1}{[L:K]}\sum_{v \in M_L} N_v \log^+|\res(f)|_v$$
and this descends to a function on $\rat_d(\ovl{K})$. We can then prove that this minimal resultant height is comparable with our moduli height on $\rat_d$. 

\begin{proposition} \label{GlobalResComparison1}
There exists constants $A,B > 0$ such that for any global field $K$,
$$A h_{\res}(f) + B \geq h(\langle f \rangle)$$
where we can take $B = 0$ if $K$ is a function field. If $K$ is a number field, then there exists constants $C,D > 0$, independent of $K$, such that 
$$C h(\langle f \rangle) + D \geq h_{\res}(f).$$
\end{proposition}

\begin{proof}
The first inequality follows from summing up Proposition \ref{LocalResComparison1} for each place $v$. When $K$ is a function field, as each place $v \in M_K$ is non-archimedean, we may take $B = 0$. For the second inequality, we use a result of Silverman \cite[Lemma 6.32]{Sil12} that tells us that for any ample height $h$ on $\Rat_d$, we have constants $C,D > 0$ such that
$$\min_{g \in \pi^{-1}(\langle f \rangle)} h(g) \leq C h(\langle f \rangle) + D$$
where $\pi: \Rat_d(\ovl{\bb{Q}}) \to \rat_d(\ovl{\bb{Q}})$ is the projection map over $\ovl{\bb{Q}}$. We may choose our height $h$ to be coming from the divisor $\Res(F) = 0$ on $\bb{P}^{2d+1}$, that is to say
$$h(f) = \sum_{v \in M_K} N_v \log |\Res(f)|^{-1}_v.$$
By definition of the minimal resultant, we clearly have 
$$\min_{g \in \pi^{-1}(\langle f \rangle)} h(g) \geq h_{\res}(f)$$
and so
$$C h(\langle f \rangle) + D \geq h_{\res}(f)$$
as desired. 
\end{proof}

\section{Arakelov--Green Functions}
\subsection{Berkovich Projective Line} We first start by recalling the definition of the Berkovich projective line \cite{BR10, Ben19} over a valued field $(k, | \cdot |)$. Given a ring $A$, we say that $| \cdot |$ is a seminorm on $A$ if it is a function $| \cdot |: A \to \bb{R}_{\geq 0}$ such that 
\begin{enumerate}
\item $|0| = 0 , |1| = 1$;
\item $|fg| \leq |f| \cdot |g|;$
\item $|f+g| \leq |f| + |g|.$
\end{enumerate}

If $0$ is the unique element satisfying $|x| = 0$, we say that $| \cdot |$ is a norm. If $|fg| = |f| \cdot |g|$, we say that it is a multiplicative (semi)norm.

\begin{definition} \label{Chap2Defn1}
Let $(k, | \cdot |)$ be an algebraically closed, complete field equipped with a norm. The Berkovich affine line $\berkA_k$ as a set consists of all multiplicative seminorms on $k[T]$ that extends the $| \cdot |$ on $k$.     
\end{definition}

Given an element $x \in \berkA_k$, we obtain a function on $\berkA_k$ by $f \mapsto x(f)$. We will write $x(f)$ as $|f(x)|$ which will be a more suggestive notation. We will put the weakest topology on $\berkA$ such that for any $f \in k[T]$, the map $f: \berkA \to \bb{R}_{\geq 0}$ given by $x \mapsto |f(x)|$ is continuous. One heuristic to think about this topology is that the continuous functions on $\berkA$ are those that can be obtained as a limit of algebraic functions (elements of $k[T])$. 
\par 
Berkovich in \cite{Ber90} classified the points of $\berkA_k$. First given any point $x \in \bb{A}^1(k) \simeq k$, we have a seminorm given by evaluating our function $f$ on $x$ and then taking its norm, i.e. $|f(x)|$. These are called the Type I points (also known as the classical points) and are dense in $\berkA_k$. In fact when $| \cdot |$ is an archimedean norm, there are no other points in $\berkA_k$ as a consequence of the Gelfand--Mazur theorem. 
\par 
We shall now assume henceforth that $| \cdot |$ is non-archimedean. Let $\zeta = D(a,r)$ be the disc of radius $r$ and centered at $a$. We can then define a seminorm on $k[T]$, which is a norm if $r > 0$, by
$$|f|_{\zeta} = \sup_{x \in \zeta} |f(x)|.$$
When $r = 0$, this is simply evaluation at the point $a$. It is easy to verify the triangle inequality but the fact that this norm is multiplicative only holds for non-archimedean fields and follows from Gauss' lemma. Let $|k^{\times}|$ denote the image of $k^{\times}$ under $| \cdot |: k \to \bb{R}_{\geq 0}$. If $r \in |k^{\times}|$, we say that $| \cdot |_{\zeta}$ is a type II point and otherwise, we say it is a type III point. 
\par 
There is a final type of point, known as type IV points, that occurs when $(k , | \cdot |)$ is not sphereically complete. Let $\zeta_1 \supset \zeta_2 \supset \cdots$ be a sequence of closed discs such that their radii $r_n$ does not go to $0$, but $\cap_{n=1}^{\infty} \zeta_n = \emptyset$. We can then define a norm
$$|f |_{\{\zeta_n\}} = \inf_{n \geq 1} |f|_{\zeta_n}.$$
Since the intersection of these discs is empty, our norm does not correspond to any of the previous types of points. It is then a theorem of Berkovich \cite{Ber90} that all points in $\berkA$ correspond to one of these four types. There is a distinguished point $\zeta_{\Gauss}$ that corresponds to the unit disc $D(0,1)$.
\par 
We can now construct the Berkovich projective line $\berkP_k$ by taking the union of $\berkA_k$ with a single point $\{\infty\}$. The space $\berkP_k$ turns out to be a compact, Haursdorff and path-connected. This gives us a suitable space to construct a non-archimedean analogue of potential theory, which has been thoroughly explored by Baker--Rumely \cite{BR10}. As this aspect of $\berkP_k$ will not be needed for us, we will refrain from mentioning more. 

\subsection{Arakelov--Green Functions} Let $(k, | \cdot|)$ be a complete algebraically closed field. Motivated by constructions of Arakelov over $\bb{C}$, for any rational map $f: \bb{P}^1_k \to \bb{P}^1_k$ of degree $d \geq 2$, Baker--Rumely \cite{BR06} associates to it an Arakelov--Green function $g_f(x,y): \berkP_k \times \berkP_k \to \bb{R} \cup \{\infty\}$ that is lower semi-contiuous. We first explain their construction. 
\par 
First, we pick a homogeneous lift $F$ of $f$ so that $F: \bb{A}^2_k \to \bb{A}^2_k$ is an endomorphism of $\bb{A}^2$. We can then define a homogeneous local height function given by
$$\Ht_F((x,y)) = \lim_{n \to \infty} \frac{1}{d^n} \frac{||F^n((x,y))||}{d^n}$$
where $||(x,y)|| = \max\{|x|,|y|\}$. This function is homogeneous, i.e. 
$$\Ht_F((\lambda x , \lambda y)) = \Ht_F((x,y)) + \log |\lambda|$$
and satisfies
$$\Ht_F(F(x,y)) = d \Ht_F((x,y)).$$
Given $x = (x_0,x_1), y = (y_0,y_1) \in \bb{A}^2_k$, we define $|x \wedge y|$ to be $|x_0 y_1 - x_1y_0|$. We now define the Arakelov--Green function as
$$g_f(x,y) = -\log |\tilde{x} \wedge \tilde{y}| + \Ht_F(\tilde{x}) + \Ht_F(\tilde{y}) - \frac{1}{d(d-1)} \log |\Res(F)|$$
where $\tilde{x}, \tilde{y}$ are lifts of $x,y$ to $\bb{A}^2$. Our definition is independent of the choices of lifts of $f,x$ and $y$. A priori it is only defined on $\bb{P}^1(k) \times \bb{P}^1(k)$ but one can show that it extends to an upper semi-continuous function on $\berkP_k \times \berkP_k$. 
\par 
We require a few properties of $g_f(x,y)$ established by Baker and Rumely. The first is that if $K$ is a global field, then for each $v \in M_K$ we obtain an Arakelov-Green function $g_{f,v}(x,y)$. Given $z_1,\ldots,z_n \in \bb{P}^1(K)$, we have the formula 
\begin{equation} \label{eq: GlobalArakelovSum1}
\sum_{v \in M_K} \sum_{i\not = j} N_v g_{f,v}(z_i,z_j) = (n-1) \sum_{i=1}^{n} \h_f(z_i).
\end{equation}

The second property is that when either $| \cdot |$ is archimedean or $| \cdot |$ is non-archimedean and $f$ has bad reduction, we have 
$$g_f(x,x) > 0 \text{ for all } x \in \berkP_k.$$
This was proven by Baker \cite{Bak09} where this is trivial to verify for type I points as $-\log |x-y|$ blows up as $y \to x$, but is not clear for the non-classical points. If $f$ has potential good reduction, we have $g_f(x,y) \geq 0$ for all $x,y \in \berkP_k$.
\par 
The third property is that $g_{\vphi \circ f \circ \vphi^{-1}}$ and $g_f$ are related by the transformation 
\begin{equation} \label{eqref: ArakelovGreenMobius1}
g_{\vphi \circ f \circ \vphi^{-1}}(\vphi(x),\vphi(y)) = g_f(x,y).
\end{equation}
In particular many properties that we are interested in for $g_f(x,y)$ are invariant under conjugation and we may freely replace $f$ with its conjugate. This will be used later in the section. 
\par 
Finally we need a lower bound that was established by Baker in \cite{Bak06}. For a complete algebraically closed field $(k, | \cdot |)$, given $z_1,\ldots,z_N \in \bb{P}^1(k)$, we have the inequality 
$$\sum_{i \not = j} g_f(z_i,z_j) \geq - O_f(N \log N).$$
It will be important for us to figure out the exact dependence on $f$ and we will do so now. After fixing a lift $F$ of $f$, we let
$$K_F = \{z \in \bb{A}^2_k \mid ||F^n(z)|| \not \to \infty\}$$
and we let $R(F)$ be the smallest positive real number such that $K_F \subseteq \{ ||z|| \leq R(F)\}$. For $F = (P(x,y),Q(x,y))$, we set $|F| = \max\{|a_i|, |b_j|\}$ where $a_i,b_j$ are the coefficients of $P$ and $Q$ respectively. 

\begin{proposition} \label{BoundRadius1}
Let $F$ be a lift such that $|F| = 1$. Then if $k$ is non-archimedean, then we may take $R(F) = |\Res(F)|^{-1/(d-1)}$. Otherwise, there exists a constant $C$ depending only on $d > 0$ such that we may take $R(F) = C|\Res(F)|^{-1/(d-1)}$. If $k$ is non-archimedean, we may take $C = 1$. 
\end{proposition}

\begin{proof}
We have to give an explicit constant $A > 0$, depending on $F$, such that 
$$A(||z||)^d \leq ||F(z)||$$
for all $z \in K^2$. Then we may take $R(F) = \max\{1,A^{-1/(d-1)}\}$ as if $||z|| \geq (1+\delta)\max\{1,A^{-1/(d-1)}\}$, then we have 
$$||F(z)|| \geq (1+\delta)^d \max\{1,A^{-1/(d-1)}\}$$
and by induction, we must have $||F^n(z)|| \to \infty$. Now, \cite[Remark 3.3]{Bak09} implies that for non-archimedean $k$, we may take $A = |\Res(F)|$. We shall now handle the archimedean case. By homogenity, we may assume that $||z|| = 1$ and it suffices to find $A$ such that $||F(z)|| \geq A$. Then if $||z|| > A^{-1/(d-1)}$, it follows easily that $||F^n(z)|| \to \infty$. We find polynomials $g_1,g_2,h_1,h_2$ and write 
$$g_1(x,y) F_1(x,y) + g_2(x,y) F_2(x,y) = \Res(F) x^{2d-1}$$
$$h_1(x,y) F_1(x,y) + h_2(x,y) F_2(x,y) = \Res(F) y^{2d-1}.$$
The coefficients of $g_i,h_j$ are polynomials with integer coefficients in the coefficients of $F_1,F_2$. In particular there exists some constant $A'$ such that $||g_i||,||h_j|| \leq A'$ as $||F|| \leq 1$. Using the triangle equality again and the fact that $||z|| = ||(x,y)|| = 1$, we obtain 
$$||F(x,y)|| \geq \frac{\Res(F)}{2A'} ||(x^{2d-1}, y^{2d-1})|| \geq \frac{\Res(F)}{2A'}$$
and so we may take $A = \frac{\Res(F)}{2A'}$ as desired. 
\end{proof}

We now obtain the following lower bound.

\begin{proposition} \label{ArakelovSumLowerBound1}
Let $z_1,\ldots,z_N \in \bb{P}^1(k)$ and assume that $N = d^k$. We have the lower bound 
$$\sum_{i \not = j} g_f(z_i,z_j) \geq -(C_k + \log |\res(f)|^{-1}) O_d (N \log N)$$
where $C_k = 1$ if $(k, | \cdot |)$ is archimedean and $0$ otherwise.
\end{proposition}

\begin{proof}
By \cite[Corollary 2.3]{Bak09}, for any lift $F$ of $f$ we have 
$$\sum_{i \not = j} g_f(z_i,z_j) \geq - \eps_k N \log N - (2 \log |\Res(F)| + r(F)) O_d(N \log N)$$
where $r(F)$ is any positive real number such that $K_F \subseteq \{ ||z|| \leq r(F)\}$. We first replace $f$ with a conjugate so that $2 |\Res(f)| \geq |\res(f)|$, which we may do so by \eqref{eqref: ArakelovGreenMobius1}. Then picking a lift $F$ such that $|F| = 1$, we have $|\Res(F)| = |\Res(f)|$ and by Proposition \ref{BoundRadius1}, we may take $R(F) = C'_k |\Res(f)|^{-1/(d-1)}$ where $C'_k = 1$ if $k$ is non-archimedean. 
\par 
Plugging it in to the above inequality and absorbing constants, since $|\Res(f)| \leq \frac{1}{2} |\res(f)|$, we obtain 
$$\sum_{i \not = j} g_f(z_i,z_j) \geq -(C_k + \log |\res(f)|)^{-1} O_d(N \log N)$$
where $C_k = 1$ if $k$ is archimedean and $0$ otherwise as desired.
\end{proof}

We note it is possible to generalize Proposition \ref{ArakelovSumLowerBound1} to higher dimensions by the work of Looper \cite{Loo24} and Ingram \cite{Ing22}. Now if $K$ is a global field and $M \subseteq M_K$ an arbitrary subset of places, we may sum over all places of $K$ to obtain the following bound.

\begin{proposition} \label{ArakelovSumLowerBound2}
Let $K$ be a global field. Then for a rational map $f: \bb{P}^1 \to \bb{P}^1$ of degree $d \geq 2$ and $z_1,\ldots,z_N \in \bb{P}^1(K)$, for any subset $M \subseteq M_K$, we have 
$$\sum_{v \in M} \sum_{i \not = j} N_v g_f(z_i,z_j) \geq (-C_K + h_{\res}(f)) O_d(N \log N)$$
where $C_K = 1$ if $K$ is a number field and is $0$ if $K$ is a function field.
\end{proposition}

\begin{proof}
This follows immediately by summing up the inequalities obtained for each place $v$ from Proposition \ref{ArakelovSumLowerBound1}, where any place of potential good reduction has a lower bound of $\geq 0$.     
\end{proof}

\section{Degeneration of Sequences of Rational Maps}
\subsection{The archimedean case} \label{sec: Archimedean}
We will construct our degeneration of sequences of rational maps over non-archimedean fields, following Favre--Gong \cite{FG24}. We first recall their construction for sequences over $\bb{C}$. We first start with the notion of an ultraftiler $\omega$. An ultrafilter is a subset of $P(\bb{N})$ satisfying:

\begin{enumerate}
\item If $A,B \in \omega$ then $A \cap B \in \omega$;
\item $\bb{N} \in \omega$ and $\emptyset \not \in \omega$;
\item For any $A \subseteq \bb{N}$, either $A$ or $\bb{N} \setminus A$ is $\in \omega$.
\end{enumerate}

The set of all ultrafilters is denoted by $\beta \bb{N}$. The topology of $\beta \bb{N}$ is generated by an open set $U_E = \{ \omega \mid E \in \omega\}$ for each subset $E \in \beta \bb{N}$. This makes $\beta \bb{N}$ into a compact Hausdorff topological space. Given a natural number $n$, we can define an ultraftiler, which we will denote by $n \in \beta \bb{N}$, by taking the collection of all subsets that contain $n$. Such a ultrafilter is called a principal ultrafilter. It is a consequence of Zorn's lemma that non-principal ultrafilters exist. A subset of $\bb{N}$ that is an element of $\omega$ will be called a $\omega$-large set.  
\par 
Given a topological space $X$ and a sequence $(x_n) \in X$, we say that $y$ is a $\omega$-limit of $x$ if for any open neighbourhood $U$ of $y$, the set $\{n \mid x_n \in U \}$ is $\omega$-large. It is a fact that if $X$ is compact, then a $\omega$-limit always exists and if $X$ is Hausdorff, the $\omega$-limit is unique. Note that for any $\omega$-large set $E$, the $\omega$-limit of $(x_n)$ depends only on the indices $n \in E$. 
\par 
Now let $(\eps_n)$ be a sequence of real numbers in $(0,1]$. We consider the Banach ring 
$$A^{\eps} = \left\{ (x_n) \in \prod_{n=1}^{\infty} \bb{C} \mid |x_n|^{\eps_n} \text{ is bounded}.\right\}$$
with the norm given by $(x_n) \mapsto \sup |x_n|^{\eps_n}$. We let $\cal{M}(A^{\eps})$ denote the Berkovich spectrum of $A^{\eps}$, which consists of all bounded multiplicative semi-norms on $A^{\eps}$, where bounded implies that there is a constant $C$ such that our multiplicative semi norm $| \cdot |$ is $\leq C || \cdot ||$ where $|| \cdot ||$ is the norm we chose for our Banach ring. We give $\cal{M}(A^{\eps})$ the weakest topology such that $| \cdot | \mapsto |f|$ is continuous for any $f \in A^{\eps}$.  
\par 
Given an ultrafilter $\omega$, we can associate to it a bounded multiplicative semi-norm $| \cdot |_{\omega}$ on $A^{\eps}$ via 
$$|(x_n)|_{\omega} = \lim_{\omega} |x_n|^{\eps_n}$$
where we view $(|x_n|^{\eps_n})$ as a sequence of reals in $[0,1]$ and hence the $\omega$-limit always exists. It turns out that this gives a homeomorphism between $\beta \bb{N}$ and $\cal{M}(A^{\eps})$ \cite[Theorem 3.8]{FG24}. 
\par 
Given a point $\omega \in \cal{M}(A^{\eps})$, we get an associated residue field $\cal{H}(\omega) = A^{\eps}/\ker(\omega)$ where 
$$\ker(\omega) = \{ (x_n) \in A^{\eps} \mid \lim_{\omega} |x_n|^{\eps_n} = 0.\}$$
The field $\cal{H}(\omega)$ is algebraically closed and is complete with respect to the induced norm $| \cdot |_{\omega}$. For a natural number $n$, the field $\cal{H}(n)$ is isomorphic to $(\bb{C}, | \cdot |^{\eps_n})$. For a ultrafilter $\omega$ in general, if $\lim_{\omega} |x_n|^{\eps_n} = e$ is a real number, then $\cal{H}(\omega) \simeq (\bb{C}, | \cdot |^e)$. But if $\lim_{\omega} |x_n|^{\eps_n} = 0$, then $\cal{H}(\omega)$ is a non-archimedean field. 
\par 
We now construct the Berkovich projective line over $A^{\eps}$. Just as in the case over a field $k$, we can define $\berkA_{A^{\eps}}$ as the set of all multiplicative seminorms on $A^{\eps}[T]$ whose restriction to $A^{\eps}$ is an element of $\cal{M}(A^{\eps})$. Then there is a natural projection map $\pi: \berkA_{A^{\eps}} \to \cal{M}(A^{\eps})$ whose preimage is exactly $\berkA_{\cal{H}(\omega)}$. Then we can construct $\berkP_{A^{\eps}}$ by adding $\{\infty\}$ to each fiber so that each fiber is now $\berkP_{\cal{H}(\omega)}$. 
\par 
Every homogeneous degree $d$ polynomial $P(z_0,z_1) \in A^{\eps}[z_0,z_1]$ defines for us a function $\psi_P$ on $\berkP_{A^{\eps}}$ by 
$$| \cdot |\mapsto \frac{|P(z_0,z_1)|}{\max\{|z_0|^d, |z_1|^d\}}$$
and the topology on $\berkP_{A^{\eps}}$ is the weakest topology such that every $\psi_P$ is a continuous function. Note that the induced topology on each fiber $\berkP_{\cal{H}(\omega)}$ is the same as its Berkovich topology. 
\par 
We thank Gong for the following argument that $\berkP_{A^{\eps}}$ is compact. It is covered by two copies of the unit discs $\ovl{\bb{D}}_{A^{\eps}}$ and each unit disc is homeomorphic to $\cal{M}(A^{\eps}[T])$ where we give $A^{\eps}[T]$ the Gauss norm, i.e. 
$$|\sum_{i=1}^{n} c_i T^i| = \max_{1 \leq i \leq n} |c_i|.$$
Then \cite[Theorem 1.2.1]{Ber90} tells us that $\cal{M}(A^{\eps}[T])$ is compact and hence $\berkP_{A^{\eps}}$ is compact as it is the union of two compact sets. 
\par 
Now given a sequence of rational maps $(f_n)$ over $\bb{C}$ of degree $d \geq 2$, we may pick lifts $F_n = (P_n,Q_n)$ such that $|F_n| = 1$. Then there exists a constant $C_d> 0$ depending only on $d$ such that $\Res(F_n) \leq C_d$. We now choose 
$$\eps_n = \left(-\log \frac{|\Res(F_n)|}{eC_d} \right)^{-1}$$ 
so that $0 \leq e_n \leq 1$. Then the sequences $P = (P_n),Q = (Q_n)$ gives us homogeneous degree $d$ polynomials over $A^{\eps}$, where the coefficient of $z^i$ is the sequence given by the coefficients of $z^i$ for $P_n$ and $Q_n$, and similarly we obtain $F = (F_n)$. Observe that 
$$|\Res(F)| = \lim_{\omega} |\Res(F_n)|^{\eps_n} \geq \bf{e}^{-1}$$
and so $\Res(F) \in (A^{\eps})^{\times}$ and in particular $F$ is the lift of some rational map $f$ on $\berkP_{A^{\eps}}$ of degree $d$. For any $\omega \in \cal{M}(A^{\eps})$, we get an induced map $f_{\omega}$ on $\berkP_{\cal{H}(\omega)}$ which can be viewed as the $\omega$-limit of $(f_n)$ for the parameters $(\eps_n)$. 
\par 
Now assume that $(f_n)$ is a sequence of rational maps over $\bb{C}$ such that $\Res(f_n) = \res(f_n)$. Then \cite[Theorem 3.17]{FG24} tells us that if $\lim_{\omega} |x_n|^{\eps_n} = 0$, we necessarily have $|\res(f_{\omega})| < 1$ and hence $f_{\omega}$ has bad reduction for $\cal{H}(\omega)$, i.e. there is no conjugate of $f_{\omega}$ that extends to a morphism over the valuation ring $O_{\cal{H}(\omega)} = \{x \in \cal{H}(\omega) \mid |x|_{\omega} \leq 1\}.$ 

\subsection{The Non-Archimedean Case} \label{sec: NonArchimedean}
We now move to the setting of non-archimedean fields. We will not restrict ourselves to just a sequence of rational maps over a fixed field $K$ but instead also let the base field vary. Let $(K_n, | \cdot |_n)$ be a sequence of non-archimedean fields which are complete and algebraically closed and let $\eps_n > 0$ be a sequence of positive reals. Observe that we allow $\eps_n$ to be $> 1$ as for non-archimedean fields, $| \cdot |^{\eps}$ still remains a norm for any positive real $\eps$ unlike the case for archimedean fields where we require $e \leq 1$ for the triangle inequality to hold. 
\par 
As in the archimedean case, we define 
$$A^{\eps} = \left\{ (x_n) \in \prod_{n=1}^{\infty} K_n \mid |x_n|^{\eps_n} \text{ is bounded.} \right\}$$
By \cite[Proposition 1.2.3]{Ber90}, we still have $\cal{M}(A^{\eps}) \simeq \beta \bb{N}$ where the map is again
$$\omega \mapsto \left((x_n) \mapsto \lim_{\omega} |x_n|^{\eps_n} \right).$$
Given $\omega \in \cal{M}(A^{\eps})$, we obtain a residue field $\cal{H}(\omega) = A^{\eps}/\ker(\omega)$. 

\begin{proposition} \label{NonArchDegen1}
The field $\cal{H}(\omega)$ is non-archimedean with respect to the induced metric $| \cdot |_{\omega}$. It is also complete and algebraically closed.
\end{proposition}

\begin{proof}
For $m \in \bb{N}$, we clearly have 
$$\lim_{\omega} |m|^{\eps_n} \leq \sup |m|^{\eps_n} \leq 1$$
as each $K_n$ is non-archimedean. Hence $| \cdot |_{\omega}$ is non-archimedean. Completeness follows from the fact that $\ker(\omega)$ is a closed subspace of $A^{\eps}$. To show $\cal{H}(\omega)$ is algebraically closed, we pick a monic polynomial $P_{\omega}[z] \in \cal{H}(\omega)[z]$. We may lift it to a monic polynomial $P$ over $A^{\eps}$ and hence a sequence of monic polynomials $(P_n)$ where $P_n = z^d + a_{n,1} z^{d-1} + \cdots + a_{n,d}$. We may factorize each $P_n$ as $\prod_{i=1}^{d} (z - \alpha_{n,i})$ where $\alpha_{n,i} \in K_n$. Then the sequence $(\alpha_{n,1})$ is an element of $A^{\eps}$ as by Newton polygons, we must have  
$$|\alpha_{n,1}| \leq \max_{1 \leq j \leq d} |a_{n,j}|^{1/j}$$
and so $|\alpha_{n,1}|^{\eps_n}$ is bounded. It follows that $(\alpha_{n,1})$ is a zero of $P(z)$ and hence is a zero of $P_{\omega}[z]$ as desired.
\end{proof}

We record some examples of $\cal{H}(\omega)$ for some sequences of fields $(K_n)$ which we are not necessarily algebraically closed. Let's say for all $n$ we have $K_n = K$ is a  field with residue characteristic $p > 0$ and that $\eps_n \to \infty$. Then $\displaystyle \lim_{\omega} |p|^{\eps_n} = 0$ and so $\cal{H}(\omega)$ has characteristic $p$.

\begin{proposition} \label{Example1}
Assume that the valuation of $K$ is discrete and that the residue field $k = O_K/\fra{m}$ is finite. Then if $\eps_n \to \infty$, for a non-principal ultrafilter $\omega$, we have $\cal{H}(\omega) = k$.
\end{proposition}

\begin{proof}
For each $x \in k$, pick a lift $\tilde{x} \in O_K$. Then we can form the constant sequence $(\tilde{x}, \tilde{x},\ldots)$ which maps into $\cal{H}(\omega)$, giving us an injective map $f: k \to \cal{H}(\omega)$. We claim that this is a field homomorphism. Given $x,y \in k$, let $z = x+y$. We have to check that $(\tilde{x}) + (\tilde{y}) = (\tilde{z})$ in $\cal{H}(\omega)$. But $|\tilde{z} - \tilde{x} - \tilde{y}| < 1$ and so $|\tilde{z} - \tilde{x} 
- \tilde{y}|^{\eps_n} \to 0$, which means $(\tilde{x}) + (\tilde{y}) = (\tilde{z})$ in $\cal{H}(\omega)$. Similarly if $xy = w$, we have $(\tilde{x} \cdot \tilde{y}) = (\tilde{w})$ and so we get a ring homomorphism. 
\par 
It suffices to show that $f$ is surjective. Pick any $a = (a_n) \in \cal{H}(\omega)$. Since $|a_n|^{\eps_n}$ is bounded, we must have $|a_n| \leq 1$ for all sufficiently large $n$. Now as $k$ is finite, there must exists a $\omega$-large subset $E \subseteq \bb{N}$ for which $|a_n - \tilde{x}| < 1$ for some fixed $x \in k$ and for all $n \in E$. Then $\lim_{\omega} |a_n - \tilde{x}|^{\eps_n} = 0$ and so $a = (\tilde{x})$, implying that $f$ is an isomorphism as desired. 
\end{proof}

On the other hand if $| \cdot |$ were not discrete, for example if $K = \widehat{\bb{Q}_p(p^{1/p^{\infty}})}$ is a perfectoid field and say $\eps_n = p^n$, then $\cal{H}(\omega)$ is of characteristic $p$ and the sequence $(p^{1/p^n})$ defines for us an element with norm $\frac{1}{p}$. This bears a resemblance to the tilt of $K$.  
\par 
In the archimedean case where $K_n = \bb{C}$, if say $\eps_n$ is the constant sequence $1$, then $\cal{H}(\omega)$ is simply $\bb{C}$ again, as shown by \cite[Theorem 3.17]{FG24}. This is due to the set $\{z \in \bb{C} \mid |z| \leq C\}$ being compact for any constant $C > 0$. On the other hand if $K_n = \bb{C}_p$ and $\eps_n = 1$, then $\cal{H}(\omega)$ is much larger than $\bb{C}_p$. For example, any sequence $(x_n)$ with $|x_i - x_j| = 1$ will correspond to a point in $\cal{H}(\omega)$. 
\par 
We can now form the Berkovich projective line $\berkP_{A^{\eps}}$ over $A^{\eps}$ where again we have a projection map $\pi: \berkP_{A^{\eps}} \to \cal{M}(A^{\eps})$ whose fiber $\pi^{-1}(\omega)$ is isomorphic to $\berkP_{\cal{H}(\omega)}$. 
\par 
Now let $(f_n)$ be a sequence of rational maps over $K_n$ of degree $d \geq 2$. We let $\eps_n = (-\log |\Res(f_n)|)^{-1}$, where as $|\Res(f_n)| \leq 1$, this is a positive real number. We pick lifts $F_n = (P_n,Q_n)$ such that $|F_n| = 1$. Then as in the archimedean case, we obtain a rational map $f$ of degree $d$ over $A^{\eps}$ with homogeneous lift given by $F = (P,Q)$, where $P = (P_n)$ and $Q = (Q_n)$ and $|\Res(F)| \in (A^{\eps})^{\times}$. 
\par 
Let us further assume that $|\res(f_n)| < 1$, i.e. $f_n$ has bad reduction over $K_n$, and that $2^{1/\eps_n} |\Res(f_n)| \geq  |\res(f_n)|$. In this case, we claim that $f_{\omega}$ always has bad reduction over $\cal{H}(\omega)$. 

\begin{proposition} \label{BadReductionDegeneration1}
We have $|\res(f_{\omega})| < 1$, i.e. $f_{\omega}$ has bad reduction over $\cal{H}(\omega)$.
\end{proposition}

\begin{proof}
Otherwise, there exists $\vphi_{\omega} \in \GL_2(\cal{H}(\omega))$ such that $|\Res(\vphi_{\omega} \circ f_{\omega} \circ \vphi_{\omega}^{-1})| = 1$. We may lift $\vphi_{\omega}$ to some element $\vphi \in M_2(A^{\eps})$. Since $|\det(\vphi_{\omega})| = \lim_{\omega} |\det(\vphi_n)|^{\eps_n}$, we know for an $\omega$-large set $E$, we have $\vphi_n \in \GL_2(K_n)$ for $n \in E$. Hence we may consider $\vphi_n \circ f_n \circ \vphi_n^{-1}$ and clearly we have 
$$\lim_{\omega} |\Res(\vphi_n \circ f_n \circ \vphi_n^{-1})|^{\eps_n} = |\Res(\vphi_{\omega} \circ f_{\omega} \circ \vphi_{\omega}^{-1})| = 1.$$
But we have
$$\lim_{\omega} |\Res(\vphi_n \circ f_n \circ \vphi_n^{-1})|^{\eps_n} \leq \lim_{\omega} |\res(f_n)|^{\eps_n} \leq \lim_{\omega} \left(2^{1/\eps_n} |\Res(f_n)| \right)^{\eps_n}$$
$$= 2 \cdot \lim_{\omega} |\Res(f_n)|^{\eps_n} = \frac{2}{e} < 1.$$
This is a contradiction. Hence $|\res(f_{\omega})| < 1$ and $f_{\omega}$ has bad reduction as desired.
\end{proof}

Recall that Favre--Gong defined model functions as follows. For any homogeneous polynomial $P \in A^{\eps}[z_0,z_1]$ of degree $l$, we may set
$$\log \psi_P = \log |P(z_0,z_1)| - l \log \max\{|z_0|,|z_1|\}.$$
For any finite collection of sections $\sigma \subset H^0(\bb{P}^1_{A^{\eps}}, O(l))$, we may set
$$\vphi_{\sigma} = \log \max_{P_i \in \sigma} \psi_{P_i}.$$
If the $P_i$'s do not have a common zero, then $\vphi_{\sigma}$ is a function from $\berkP_{A^{\eps}} \to \bb{R}$ and we say that $\vphi_{\sigma}$ is a model function. By \cite[Theorem 4.8]{FG24}, the model functions are dense in the space of continuous functions as we may take $b_1 = (x_n)$ where each $x_n \in K_n$ is an element satisfying $\frac{1}{2} \leq |x_n|^{\eps_n} < 1$, which always exists due to non-discreteness of $|K_n^{\times}|$.
\par 
We now make an important observation, which is that the induced norm on $\pi^{-1}(n) \simeq \berkP_{K_n}$ is not $| \cdot |_n$ but instead $| \cdot |_n^{\eps_n}$. In particular when evaluating $\log \psi_P$ on $\pi^{-1}(n)$, other than replacing $| \cdot |$ with $| \cdot |_n$, we should also add an extra scaling factor of $\eps_n$. It will be convenient to set $\eps_{\omega} = 1$ for non-principal ultrafilters, so that for all ultrafilters $\omega \in \beta \bb{N}$ we have the following relation 
$$\lim_{\omega} \eps_n \left(\log |P(z_0,z_1)|_n - l \log \max \{|z_0|,|z_1|\} \right) $$
$$= \eps_{\omega} \left(\log |P(z_0,z_1)|_{\omega} - l \log \max\{ |z_0|_{\omega},|z_1|_{\omega}\}\right).$$
\par 
We now give a criterion to check if a classical point $x \in \bb{P}^1(\cal{H}(\omega))$ is the $\omega$-limit of a sequence $(x_n) \in \bb{P}^1(K_n)$. 

\begin{proposition} \label{LimitCriterion1}
Let $x \in \bb{P}^1(\cal{H}(\omega))$ and let $(x_n) \in \bb{P}^1(K_n)$. Then $x = \lim_{\omega} x_n$ if and only if $\vphi(x) = \lim_{\omega} \vphi(x_n)$ for all model functions $\vphi$. 
\end{proposition}

\begin{proof}
Since $\berkP_{A^{\eps}}$ is compact and Hausdorff (where Hausdorff follows from the proof of \cite[Theorem 4.8]{FG24}), the $\omega$-limit of $(x_n)$ always exists and is unique. Furthermore it must lie in $\pi^{-1}(\omega) = \bb{P}^1(\cal{H}(\omega))$. Let this limit be $x'$. If $x' \not = x$, then by Uryhshon's lemma, we may find a continuous function $\psi$ such that $\psi(x) = 0$ and $\psi(x') = 1$. By density of model functions, we may replace $\psi$ with a model function and weaken the conditions to be $\psi(x) < \eps$ and $\psi(x') > 1-\eps$. But since $x' = \lim_{\omega} x_n$, we know that 
$$\psi(x') = \lim_{\omega} \psi(x_n).$$
Thus if we also had 
$$\psi(x) = \lim_{\omega} \psi(x_n),$$
this would be a contradiction as desired. 
\end{proof}

We now need the following fact about Type I points in $\pi^{-1}(\omega)$. 

\begin{proposition} \label{TypeOne1}
Let $x = [z_0:z_1] \in \bb{P}^1(\cal{H}(\omega))$ be a Type I point. Then there exists a sequence $(x_n) \in \bb{P}^1(K_n)$ such that $x = \lim_{\omega} x_n$. 
\end{proposition}

\begin{proof}
We may without loss of generality assume that $z_1 \not = 0$. Then we may take $z_1 = 1$ and lift $z_0$ to $(z_n) \in A^{\eps}$. We then set $x_n = [z_n:1] \in \bb{P}^1(K_n)$. By Proposition \ref{LimitCriterion1}, it suffices to check that 
$$\vphi(x) = \lim_{\omega} \vphi(x_n)$$
for all model functions $\vphi$. But given 
$$\psi_P = \frac{P(z_0,z_1)}{\max\{|z_0|,|z_1|\}^l},$$
it is clear that 
$$\psi_P(x) = \lim_{\omega}  \psi_P(x_n)$$
by direct evaluation and so the same holds for any model function $\vphi_{\sigma}$ as desired. 

\end{proof}

\section{Arakelov--Green Functions on $\berkP_{A^{\eps}}$} Let $(K_n, | \cdot |_n)$ be a sequence of complete algebraically closed fields where either $K_n = \bb{C}$ for all $n$ or $K_n$ is non-archimedean for all $n$. Then if $(f_n)$ is a sequence of rational maps over $(K_n)$ with homogeneous lift $(F_n) = ((P_n,Q_n))$, we can choose $\eps_n$ such that they extend to a rational map $f$ over $A^{\eps}$ with homogeneous lift $F = (P,Q)$ where $P = (P_n)$ and $Q = (Q_n)$. If we are in the non-archimedean case, we further assume that $(f_n)$ has bad reduction for all $n$. Then we may choose our lifts such that $||F_n||_n = 1$ and choose $\eps_n = (-\log |\Res(F_n)|_n)^{-1}$ for the non-archimedean case and $(-\log |\Res(F_n)|_n/eC_d)^{-1}$ for some constant $C_d$ in the archimedean case. Either way $\eps_n \log |\Res(F_n)|$ is uniformly bounded from above and below.  
\par 
With our lift $F$, we can define the Arakelov--Green function for each $\berkP_{\cal{H}(\omega)}$ by 
$$g_{f_{\omega}}(x,y) = \eps_{\omega} \left(-\log |\tilde{x} \wedge \tilde{y}|_{\omega} + \Ht_{F_{\omega}}(\tilde{x}) + \Ht_{F_{\omega}}(\tilde{y}) - \frac{1}{d(d-1)} \log |\Res(F)|_{\omega} \right)$$
for $(x,y) \in \berkP_{\cal{H}(\omega)} \times \berkP_{\cal{H}(\omega)}$,
where $\eps_{\omega} = \eps_n$ if $\omega = n \in \bb{N}$, and is $1$ otherwise. The constant $\eps_{\omega}$ is chosen so that we have the equality 
$$\lim_{\omega} \eps_n \log |x_n|_n = \eps_{\omega} \log |x|_{\omega}$$
for any $x = (x_1,\ldots,x_n,\ldots) \in A^{\eps}$. Again, recall that the induced norm on $\berkP_{K_n}$ from $\berkP_{A^{\eps}}$ is $| \cdot |_n^{\eps_n}$ and not $| \cdot |_n$. 
\par 
For convenience, let $X = \berkP_{A^{\eps}}$ and let $X^{(2)}$ be the fiber product of $\berkP_{A^{\eps}}$ with itself over $\cal{M}(A^{\eps})$. Then we can glue the various $g_{f_{\omega}}(x,y)$'s together to obtain a function
$$g_f(x,y): X^{(2)} \to \bb{R} \cup \{\infty\}$$
that is lower semicontinuous when restricted to each $\pi^{-1}(\omega)$. The main result of this section is to prove that $g_f(x,y)$ is lower semicontinuous not just fiber-wise, but on the entire $X^{(2)}$. 
\par 
Recall that Favre--Gong \cite[Lemma 4.14]{FG24} provide a simple criteria to check if a function $\vphi: X \to \bb{R}$ is continuous. 

\begin{proposition} \label{ContinuityCriterion1}
A function $\vphi: X \to \bb{R}$ is continuous if and only if for all $\omega \in \beta \bb{N}$, the restriction $\vphi \vert_{\pi^{-1}(\omega)}$ is continuous and for any sequence of Type I points $(x_n) \in \pi^{-1}(n)$, we have $\lim_{\omega} \vphi(x_n) = \vphi(\lim_{\omega} x_n).$
\end{proposition}

\begin{proof}
This is \cite[Lemma 4.14]{FG24}. Although they only prove it in the archimedean case, the same proof applies in the non-archimedean setting.    
\end{proof}

We prove a similar criteria to check lower semicontinuity. 

\begin{proposition} \label{ContinuityCriterion2}
Let $\vphi(x,y): X^{(2)} \to \bb{R} \cup \{ \infty\}$ be a function such that on each fiber, it is lower semi-continuous and that 
\begin{equation} \label{eq: SemiContinuity1}
\vphi(x,y) = \liminf_{\substack{ (x',y') \to (x,y) \\ (x',y') \text{ Type I}}} \vphi(x',y').
\end{equation}
If
\begin{equation} \label{eq: SemiContinuity2}
\lim_{\omega} \vphi(x_n,y_n) = \vphi(\lim_{\omega} x_n, \lim_{\omega} y_n)
\end{equation}
for all sequence of type I points $(x_n), (y_n)$, then $\vphi$ itself is lower semi-continuous. 
\end{proposition}

\begin{proof}
We have to show that $\vphi^{-1}((\delta, \infty))$ is open. Assume otherwise. Then there exists $(x_0,y_0)$ such that $\vphi(x_0,y_0) >\delta + \eps$, with $\eps > 0$, such that for any open neighbourhood $U \times V$, there exists $x_U,y_V$ such that $\vphi(x_U,y_V) < \delta$. We first show that we may assume $x_U,y_V \in \pi^{-1}(\bb{N})$. We do so by first assuming that $x_U$ and $y_V$ are both Type I points by \eqref{eq: SemiContinuity1}. Then we may further assume that $(x_U,y_V) \in \pi^{-1}(\bb{N})$ as any Type I point in $\pi^{-1}(\omega)$ is the $\omega$-limit of a sequence of Type I points in $\pi^{-1}(\bb{N})$ by Proposition \ref{TypeOne1} and we can then apply \eqref{eq: SemiContinuity2}.
\par 
Now since $\vphi$ is lower semicontinuous on each fiber, we may assume that $x_0,y_0 \in \pi^{-1}(\omega)$ for some non-principal ultrafilter $\omega$. Then there exists opens $U_{\omega},V_{\omega}$ of $\berkP_{\cal{H}(\omega)}$ such that $\vphi(x,y) > \delta+ \frac{\eps}{2}$ for all $(x,y) \in U_{\omega}\times V_{\omega}$. Extend $U_{\omega}$ and $V_{\omega}$ to open sets $U,V$ of $\berkP_{A^{\eps}}$, which we may do so as the topology on $\berkP_{\cal{H}(\omega)} = \pi^{-1}(\omega)$ is the subspace topology from $\berkP_{A^{\eps}}$. Let $U_n,V_n$ be the intersection of $U,V$ with $\pi^{-1}(n)$. Consider 
$$E = \{ n \in \bb{N} \mid \vphi(x,y) > \delta + \frac{\eps}{2} \text{ for all } (x,y) \in \pi^{-1}(n).\}$$
We claim that $E$ is $\omega$-large. Otherwise, $\bb{N} \setminus E$ is $\omega$-large and we may find sequences $(x_n),(y_n)$ such that for $n \in \bb{N} \setminus E$, we have 
$$\vphi(x_n,y_n) \leq \delta + \frac{\eps}{2} \implies \lim_{\omega} \vphi(x_n,y_n) = \lim \vphi( \lim_{\omega} x_n, \lim_{\omega} y_n)\leq \delta + \frac{\eps}{2}.$$
But $\lim_{\omega} x_n \in U_{\omega}$ and $\lim_{\omega} y_n \in V_{\omega}$ and so this is a contradiction. Hence $E$ is $\omega$-large and now if we let $E'$ be the open subset of $\beta \bb{N}$ that corresponds to $\{\omega \in \beta \bb{N} \mid E \in \omega\}$, we may take $U' = U \cap \pi^{-1}(E)$ and $V' = V \cap \pi^{-1}(E)$ to obtain open subsets for which on $\pi^{-1}(\bb{N})$, we always have $\vphi(x,y) > \delta + \frac{\eps}{2}$. This however contradicts the existence of $x_{U'}, y_{V'}$ and so we are done. 
\end{proof}

We now prove that $g_f(x,y)$ is lower semicontinuous. 

\begin{proposition} \label{ArakelovGreenContinuity1}
The function $g_f(x,y): X^{(2)} \to \bb{R} \cup \{\infty\}$ is lower semicontinuous.
\end{proposition}

\begin{proof}
We apply Proposition \ref{ContinuityCriterion2}. Condition \eqref{eq: SemiContinuity1} easily follows from the fact that $g_{f_{\omega}}(x,y)$ is continuous when $x$ is a fixed non-Type I point. To check condition \eqref{eq: SemiContinuity2}, we use the definition
$$g_{f_{\omega}}(x,y) = \eps_{\omega} \left(-\log |\tilde{x} \wedge \tilde{y}|_{\omega} + \Ht_{F_{\omega}}(\tilde{x}) + \Ht_{F_{\omega}}(\tilde{y}) - \frac{1}{d(d-1)} \log |\Res(F)|_{\omega} \right).$$
Clearly we have $\lim_{\omega} \eps_n \log |\Res(F)|_n = \eps_{\omega} \log |\Res(F)|_{\omega}$. Given sequences of Type I points $(x_n),(y_n)$ with $x_n,y_n \in \pi^{-1}(n)$, we may pick lifts $\tilde{x}_n, \tilde{y}_n$ such that $|\tilde{x}_n| = |\tilde{y}_n| = 1$. Then $\tilde{x} = (\tilde{x}_n), \tilde{y} = (\tilde{y}_n)$ gives us lifts for $(x_n),(y_n)$ and it is immediate to check that 
$$\lim_{\omega} \eps_n \log |\tilde{x}_n \wedge \tilde{y}_n|_n = \eps_{\omega} \log |\tilde{x} \wedge \tilde{y}|_{\omega}.$$
It suffices to check that 
$$\lim_{\omega} \eps_n \Ht_{F_n}(\tilde{x}_n) = \Ht_{F_{\omega}}(\tilde{x}).$$
We first claim that since As $\eps_{\omega} \log |\Res(F)|_{\omega}^{-1}$ is uniformly bounded from above and $||F||_{n} = 1$, there exists a uniform constant $C$ such that for all $\omega \in \beta \bb{N}$,  
$$\left|\frac{\eps_{\omega}}{d^k} \log |F^{(k)}(\tilde{x})|_{\omega} - \eps_{\omega} \Ht_{F_{\omega}}(\tilde{x}) \right| \leq \frac{C}{d^k}.$$
Indeed using the telescoping sum 
$$\eps_{\omega} \left(\Ht_{F_{\omega}} - \frac{1}{d^k} \log |F^{(k)}(\tilde{x})|_{\omega} \right)= \eps_{\omega} \sum_{n \geq k} \frac{1}{d^n} \left(\frac{1}{d} \log |F^{({n+1})}(\tilde{x})|_{\omega} - \log |F^n(\tilde{x})|_{\omega} \right),$$
by \cite[Lemma 5]{Ing22}, we may bound 
$$\left(\frac{1}{d} \log |F^{(n+1)}(\tilde{x})|_{\omega} - \log |F^n(\tilde{x})|_{\omega} \right)$$
in terms of $\eps_{\omega} |\Res(f)|_{\omega}^{-1}$ and $\eps_{\omega} ||F||_{\omega}$. As $\eps_{\omega} |\Res(f)|^{-1}_{\omega}$ is uniformly bounded and $||F||_{\omega} = 1$, both of this quantities are uniformly bounded, we obtain our uniform constant $C$. 
\par 
Now by continuity of the function 
$$\log |F^{k}(x_0,x_1)| - d^k \log \max\{|x_0|,|x_1|\},$$
we have 
$$\lim_{\omega} \eps_n \log |F^{(k)}(\tilde{x}_n)|_n = \eps_{\omega} \log |F^{(k)}(\tilde{x})|_{\omega}$$
and taking $k \to \infty$ implies that 
$$\lim_{\omega} \eps_n \Ht_{F_n}(\tilde{x}_n) = \eps_{\omega} \Ht_{F_{\omega}}(\tilde{x})$$
as desired.
\end{proof}

\section{Degeneration of Arakelov--Green's Function} We now prove Theorem \ref{IntroUniformBaker3} using Proposition \ref{ArakelovGreenContinuity1}. We first prove a general statement about open covers. 

\begin{proposition} \label{OpenCover1}
Let $\omega \in \beta \bb{N}$ and let $U$ be an open set such that it contains $\pi^{-1}(\omega)$. Then if 
$$E = \{ n \in \bb{N} \mid \pi^{-1}(n) \subseteq U\},$$
the set $E$ is $\omega$-large. 
\end{proposition}

\begin{proof}
Assume otherwise that $E$ is not $\omega$-large. Then $\bb{N} \setminus E$ is $\omega$-large and we may find a sequence $(x_n)$ such that $x_n \in \pi^{-1}(n)$ but $x_n \not \in U$ for $n \in \bb{N} \setminus E$. Then using the general definition of $\omega$-limit, as $\berkP_{A^{\eps}}$ is a compact and Hausdroff space, the $\omega$-limit $\lim_{\omega} x_n$ exists. Let this limit be $x$. Then since $\pi: \berkP_{A^{\eps}} \to \cal{M}(A^{\eps}) \simeq \beta \bb{N}$ is continuous, we must have $\lim_{\omega} x_n \in \pi^{-1}(\omega)$. If $x \in U$, then by definition of being a $\omega$-limit, we must have 
$$E = \{n \mid x_n \in U\}$$
be $\omega$-large but we assumed otherwise. Hence $x \not \in U$ but $U$ contains $\pi^{-1}(\omega)$ which is a contradiction. 
\end{proof}

We first handle the archimedean case. 

\begin{theorem} \label{UniformBakerLower1}
There exists a constant $\delta > 0$ depending on $d$ and a positive integer $N$ such that for any degree $d \geq 2$ rational map $f$ over $(\bb{C}, | \cdot |)$ with $| \cdot |$ the standard norm, we may cover $\bb{P}^1(\bb{C})$ into $N$ open sets $U_1,\ldots,U_N$ such that 
$$g_{f}(x,y) > \delta \max\{-\log |\res(f)|,1\}$$ 
for any $x,y \in U_i$.  
\end{theorem}

\begin{proof}
Since 
$$g_{\vphi \circ f \circ \vphi^{-1}}(\vphi(x),\vphi(y)) = g_f(x,y)$$
for any $\vphi \in \GL_2(\bb{C})$, the existence of $N$ many such open sets $U_i$ is invariant under conjugation and so we may freely conjugate $f$. Assuming otherwise, then for each pair $(N,\delta)$, we may find a rational map $f$ defined over $\bb{C}$ such that there does not exists $N$ open sets $U_1,\ldots,U_N$ with 
$$g_f(x,y) \geq \delta \max\{-\log |\res(f)|,1\} \text{ for all }x,y \in U_i.$$
Taking $(N,\delta) = (n, \frac{1}{n})$, we may find a sequence of fields $(K_n, | \cdot |_n)$ with $(K_n, | \cdot |_n) = (\bb{C},| \cdot |)$, along with rational maps $f_n$ defined over $K_n$ such that for each $n$, there does not exists $n$ open sets $U_1,\ldots,U_n$ such that 
$$g_{f_n}(x,y) > \frac{1}{n} \max\{-\log |\res(f_n)|,1\} \text{ for all } x,y \in U_i.$$
We may replace $f_n$ by its conjugate so that $|\Res(f_n)| = |\res(f_n)|$. We now consider degenerating the sequence as in Section \ref{sec: Archimedean}. Taking 
$$\eps_n = \left(-\log \frac{|\Res(F_n)|}{eC_d} \right)^{-1},$$
we obtain a rational map $f$ of degree $d$ on $\berkP_{A^{\eps}}$. Then for any non-principal ultrafilter $\omega$, we know that either $\cal{H}(\omega)$ is archimedean or $\cal{H}(\omega)$ is non-archimedean and $f_{\omega}$ has bad reduction. Either way by \cite[Theorem 1.12]{Bak09}, we know that $g_f(\zeta,\zeta) > 0$ for any $\zeta \in \pi^{-1}(\omega)$. We now apply Proposition \ref{ArakelovGreenContinuity1} to conclude that for each such $\zeta$, we may find a $\delta_{\zeta} > 0$ and an open neighbourhood $U_{\zeta}$ such that $g_f(x,y) > \delta_{\zeta}$ for all $x,y \in U_{\zeta}$. By compactness of $\pi^{-1}(\omega)$, we may find $U_1,\ldots,U_N$ and a $\delta > 0$ such that $g_f(x,y) > \delta$ for all $x,y \in U_i$ and $\pi^{-1}(\omega) \in \bigcup_{i=1}^{N} U_i$. 
\par 
By Proposition \ref{OpenCover1}, the set 
$$E = \{n \mid \pi^{-1}(n) \subseteq \bigcup_{i=1}^{N} U_i\}$$
is $\omega$-large and thus contain arbitrarily large natural numbers $n$ as $\omega$ is non-principal. Moving back to the field $(K_n, | \cdot |_n)$, this implies we have $N$ open sets $U_{i,n}$ covering $\bb{P}^1(K_n)$ such that for all $x,y \in U_{i,n}$, we have 
$$g_{f_n}(x,y) > \delta (- \log |\res(f_n)|_n + C_d) \geq \delta \max\{- \log |\res(f)|,1\}$$
as $C_d$ is taken large enough such that $-\log |\res(f)| + C_d \geq 1$ for all $f \in \Rat_d(\bb{C})$. But if $n > N, \delta^{-1}$, this is a contradiction and hence we are done. 
\end{proof}

We next do the non-archimedean case, which follows in a similar fahsion.

\begin{theorem} \label{UniformBakerLower2}
There exists a constant $\delta > 0$ depending only on $d$ and a positive integer $N$ such that for complete complete algebraically closed field $(K, | \cdot |)$ where $| \cdot |$ is non-archimedean and any degree $d \geq 2$ rational map $f$ with bad reduction, we may cover $\berkP_K$ with $N$ open sets $U_1,\ldots,U_N$ such that 
$$g_f(x,y) > -\delta \log |\res(f)|$$
for all $x,y \in U_i$. 
\end{theorem}

\begin{proof}
Assume otherwise. Then for each natural number $n$, we have a non-archimedean field $(K_n, | \cdot |_n)$ with a rational map $f_n$ of degree $d \geq 2$ for which there does not exists $n$ opens $U_1,\ldots,U_n$ covering $\berkP_{K_n}$ such that 
$$g_{f_n}(x,y) \geq -\frac{1}{n} \log |\res(f_n)|_n.$$
Take $\eps_n = -\log|\res(f_n)|_n$. Again we may replace $f_n$ by any conjugate and so we may assume that $2^{1/\eps_n} |\Res(f_n)| \geq |\res(f_n)|$. Now we degenerate in Section \ref{sec: NonArchimedean}. We obtain a rational map $f$ of degree $d$ over $A^{\eps}$ such that for any $\omega$, the map $f_{\omega}$ has bad reduction over $\cal{H}(\omega)$. The rest follows similarly to the proof of Theorem \ref{UniformBakerLower1} where we again use \cite[Theorem 3.14]{Bak09} to obtain $N$ open sets $U_i$ covering $\pi^{-1}(\omega)$ and a $\delta > 0$ for which 
$$g_{f_{\omega}}(x,y) > \delta \text{ for all } x,y \in U_i.$$
We then arrive at a contradiction similarly as desired. 
\end{proof}

\section{Proofs of Theorems \ref{IntroUniformBaker1} and \ref{IntroUniformBaker2}}
We now prove our main theorems Theorem \ref{IntroUniformBaker1} and Theorem \ref{IntroUniformBaker2}. 

\begin{theorem} \label{UniformBakerTheorem1}
Let $K$ be a number field and $f: \bb{P}^1 \to \bb{P}^1$ a rational map of degree $d \geq 2$ with $s$ places of bad reduction, where we include all archimedean places. There exists constants $c_1,c_2 > 0$, depending only on $d$ and independent of $f$ and $K$, such that 
$$
\# \left\{ x \in \bb{P}^1(L) \mid \h_f(x) \leq \frac{c_1}{s} \max\{h_{\rat_d}(\langle f \rangle) ,1\}\right\} \leq c_2 s \log(s).
$$
\end{theorem}

\begin{proof}
We assume otherwise that we have $c_2 s \log s$ many points in our set and then will choose $c_2$ accordingly later. We split into two cases depending whether $h_{\res}(f)$ is $\geq 1$ or not. 
\par 
First let's assume that $h_{\res}(f) \geq 1$. Then there exists a place $w$ of $L$ such that 
\begin{equation} \label{eq: Resultant1}
-N_w \log |\res(f)|_w \geq \frac{h_{\res}(f)}{s}.
\end{equation}
By Theorems \ref{UniformBakerLower1} and \ref{UniformBakerLower2}, there exists $N$, depending only on $d$, such that we can cover $\berkP_v$ with $N$ opens $U_i$ for which 
$$g_{f,w}(x,y) > \delta (-\log |\res(f)|_v)$$
for any $x,y$ that lie in the same open $U_i$. Then by pigeonhole, we may find $M = \frac{c_2 (s \log s)}{N}$ many such points that lie in the same open set $U_i$, say $z_1,\ldots,z_M$. Now we have 
\begin{equation} \label{eq: SumArakelov1}
\sum_{v \in M_K} \sum_{i \not = j} N_v g_{f,v}(z_i,z_j) = M \sum_{i=1}^{M} \h_f(z_i) \leq \frac{c_1 M^2}{s} h_{\res}(f).
\end{equation}
On the other hand if we let $M'_K = M_K \setminus v$, by Proposition \ref{ArakelovSumLowerBound1} where we absorb $C_K$ into $h_{\res}(f)$ as we assumed $h_{\res}(f) \geq 1$, we have 
$$\sum_{v \in M'_K} \sum_{i \not = j} N_v g_{f,v}(z_i,z_j) \geq -(h_{\res}(f)) O_d(M \log M).$$
Now adding in the contribution for the place $w$, as $g_{f,w}(z_i,z_j) > \delta (-\log |\res(f)|_v)$, we obtain 

$$\sum_{v \in M_K} \sum_{i \not = j} N_v g_{f,v}(z_i,z_j) \geq (-h_{\res}(f)) O_d(M \log M) + M(M-1) \delta (-\log |\res(f)|_v).$$
Using \eqref{eq: Resultant1}, we get 
$$\sum_{v \in M_L} \sum_{i \not = j} N_v g_{f,v}(z_i,z_j) \geq \left( \frac{M(M-1) \delta}{s} - O_d(M \log M) \right) h_{\res}(f).$$
By \eqref{eq: SumArakelov1}, we get
$$\frac{c_1 M^2}{s} \geq \frac{M(M-1) \delta }{s} - O_d(M \log M).$$
We thus obtain a contradiction if $c_1 \leq \frac{\delta}{2}$ and $M \geq c_3 s \log s$ for some constant $c_3 > 0$ depending only on $d$. This leads to a contradiction if we have more than $c_3 N s \log (s)$ and so we may take $c_2 = c_3 N$ as desired.
\par 
Now if $h_{\res}(f) \leq 1$, we instead pick an arbitrary archimedean place $w$. We apply Theorem \ref{UniformBakerLower1} to again get $N$ opens $U_i$ covering $\bb{P}^1(\bb{C})$ such that 
$$g_{f,w}(x,y) > \delta.$$
We then similarly get 
$$\sum_{v \in M_L} \sum_{i \not = j} N_v g_{f,v}(z_i,z_j) \geq \left( \frac{M(M-1) \delta}{[K:\bb{Q}]} - O_d(M \log M) \right) h_{\res}(f)$$
which again by \eqref{eq: SumArakelov1} gives us
$$\frac{c_1 M^2}{s} \geq \frac{M(M-1)\delta}{2s}  - O_d(M \log M)$$
as $2s \geq [K:\bb{Q}]$ since we included all archimedean places. Again for $c_1 \leq \frac{\delta}{2}$, we obtain a contradiction for $M$ sufficiently large depending only on $d$ as desired. 

\end{proof}

By Proposition \ref{GlobalResComparison1}, we may replace $h_{\res}(f)$ with $h_{\rat_d}( \langle f \rangle)$ in Theorem \ref{UniformBakerTheorem1} which gives us Theorem \ref{IntroUniformBaker1}. This immediately implies there are at most $c_2 s \log(s)$ many preperiodic points in $K$, which generalizes Benedetto's result \cite{Ben07} for polynomials.  As a simple corollary, we obtain a lower bound on the canonical height of a non-preperiodic point which generalizes Looper's result \cite{Loo19} for polynomials. 

\begin{theorem} \label{UniformCanonicalHeight1}
Let $K$ be a number field and $f: \bb{P}^1 \to \bb{P}^1$ a rational map of degree $d \geq 2$ with $s$ places of bad reduction. Then there exists a constant $c_3 > 0$, depending only on $d$ and independent of $K$ and $f$, such that for $x \in \bb{P}^1(K)$, we have 
$$\h_f(x) = 0 \text{ or } \h_f(x) \geq \frac{1}{d^{c_3 s \log s}} \max\{1, h_{\rat_d}(\langle f \rangle)\}.$$
\end{theorem}

\begin{proof}
Assume otherwise that we have a $x$ satisfying 
$$\h_f(x) \leq \frac{1}{d^{c_3 s \log s}} \max\{1,h_{\rat_d}(\langle f \rangle)\}.$$
Then $x,f(x),\ldots,f^n(x)$ for $n \leq \frac{c_3}{2} s \log(s)$ all satisfy
$$\h_f(f^i(x)) \leq \frac{c_1}{s} \max\{1, h_{\rat_d}(\langle f \rangle)\}$$
if $\frac{c_3}{2} \geq c_1$. This contradicts Theorem \ref{UniformBakerTheorem1}. 
\end{proof}

This answers a conjecture of Silverman \cite[Conjecture 4.98]{Sil07} under the assumption that $f$ has at most $s$ many places of bad reduction. 
\par 
For function fields, we obtain an analogous result.

\begin{theorem} \label{UniformBakerTheorem2}
Let $K$ be a function field and $L$ a finite extension. Let $f: \bb{P}^1 \to \bb{P}^1$ be a rational map of degree $d$ defined over $L$ with $s$ places of bad reduction. Assume that $f$ is not isotrivial. Then there exists $c_1,c_2 > 0$, depending only on $d$, such that 
$$\# \left\{ x \in \bb{P}^1(L) \text{ with } \h_{f}(x)\leq \frac{c_1}{s} h_{\res}(f)  \right\} \leq  c_2 (s \log s).$$
\end{theorem}

\begin{proof}
The proof of this is exactly the same as the proof of Theorem \ref{UniformBakerTheorem1} in the case where $h_{\res}(f) \geq 1$. As $f$ is not isotrivial, by \cite[Theorem 1.9]{Bak09} $f$ has at least one place of bad reduction and so that $h_{\res}(f) > 0$. Then there exists a place $w$ of $L$ such that 
$$-N_w \log |\res(f)|_w \geq \frac{h_{\res}(f)}{s}.$$
The rest follows similarly as in Theorem \ref{UniformBakerTheorem1}.
\end{proof}

\printbibliography

@article {HS88,
    AUTHOR = {Hindry, M. and Silverman, J. H.},
     TITLE = {The canonical height and integral points on elliptic curves},
   JOURNAL = {Invent. Math.},
  FJOURNAL = {Inventiones Mathematicae},
    VOLUME = {93},
      YEAR = {1988},
    NUMBER = {2},
     PAGES = {419--450},
}

@article {DF24,
    AUTHOR = {Doyle, John R. and Faber, Xander},
     TITLE = {New families satisfying the dynamical uniform boundedness
              principle over function fields},
   JOURNAL = {Math. Ann.},
  FJOURNAL = {Mathematische Annalen},
    VOLUME = {388},
      YEAR = {2024},
    NUMBER = {1},
     PAGES = {985--1000},
}

@article {DP20,
    AUTHOR = {Doyle, John R. and Poonen, Bjorn},
     TITLE = {Gonality of dynatomic curves and strong uniform boundedness of
              preperiodic points},
   JOURNAL = {Compos. Math.},
  FJOURNAL = {Compositio Mathematica},
    VOLUME = {156},
      YEAR = {2020},
    NUMBER = {4},
     PAGES = {733--743},
}

@book {LP24,
    AUTHOR = {Lemanissier, Thibaud and Poineau, J\'er\^ome},
     TITLE = {Espaces de {B}erkovich globaux---cat\'egorie, topologie,
              cohomologie},
    SERIES = {Progress in Mathematics},
    VOLUME = {353},
 PUBLISHER = {Birkh\"auser/Springer, Cham},
      YEAR = {[2024] \copyright 2024},
     PAGES = {vii+289},
}

@article {Luo22,
    AUTHOR = {Luo, Yusheng},
     TITLE = {Trees, length spectra for rational maps via barycentric
              extensions, and {B}erkovich spaces},
   JOURNAL = {Duke Math. J.},
  FJOURNAL = {Duke Mathematical Journal},
    VOLUME = {171},
      YEAR = {2022},
    NUMBER = {14},
     PAGES = {2943--3001},
}

@misc{Loo21b,
      title={The Uniform Boundedness and Dynamical Lang Conjectures for polynomials}, 
      author={Nicole R. Looper},
      year={2021},
      eprint={2105.05240},
      archivePrefix={arXiv},
      url={https://arxiv.org/abs/2105.05240}, 
}

@article {Loo21,
    AUTHOR = {Looper, Nicole R.},
     TITLE = {Dynamical uniform boundedness and the {$abc$}-conjecture},
   JOURNAL = {Invent. Math.},
  FJOURNAL = {Inventiones Mathematicae},
    VOLUME = {225},
      YEAR = {2021},
    NUMBER = {1},
     PAGES = {1--44},
}

@article {DeM08,
    AUTHOR = {DeMarco, Laura G. and McMullen, Curtis T.},
     TITLE = {Trees and the dynamics of polynomials},
   JOURNAL = {Ann. Sci. \'Ec. Norm. Sup\'er. (4)},
  FJOURNAL = {Annales Scientifiques de l'\'Ecole Normale Sup\'erieure.
              Quatri\`eme S\'erie},
    VOLUME = {41},
      YEAR = {2008},
    NUMBER = {3},
     PAGES = {337--382},
}

@article {Dem07,
    AUTHOR = {DeMarco, Laura},
     TITLE = {The moduli space of quadratic rational maps},
   JOURNAL = {J. Amer. Math. Soc.},
  FJOURNAL = {Journal of the American Mathematical Society},
    VOLUME = {20},
      YEAR = {2007},
    NUMBER = {2},
     PAGES = {321--355},
}

@article {Son23,
    AUTHOR = {Song, Yinchong},
     TITLE = {Asymptotic behavior of the {Z}hang-{K}awazumi's
              {$\varphi$}-invariants},
   JOURNAL = {Adv. Math.},
  FJOURNAL = {Advances in Mathematics},
    VOLUME = {435},
      YEAR = {2023},
     PAGES = {Paper No. 109349, 45},
}

@misc{Yua24,
      title={Arithmetic bigness and a uniform Bogomolov-type result}, 
      author={Xinyi Yuan},
      year={2024},
      eprint={2108.05625},
      archivePrefix={arXiv},
      url={https://arxiv.org/abs/2108.05625}, 
}

@article {Fal21,
    AUTHOR = {Faltings, Gerd},
     TITLE = {Arakelov geometry on degenerating curves},
   JOURNAL = {J. Reine Angew. Math.},
  FJOURNAL = {Journal f\"ur die Reine und Angewandte Mathematik. [Crelle's
              Journal]},
    VOLUME = {771},
      YEAR = {2021},
     PAGES = {65--84},
}

@article {Rum15,
    AUTHOR = {Rumely, Robert},
     TITLE = {The minimal resultant locus},
   JOURNAL = {Acta Arith.},
  FJOURNAL = {Acta Arithmetica},
    VOLUME = {169},
      YEAR = {2015},
    NUMBER = {3},
     PAGES = {251--290},
}

@article {STW14,
    AUTHOR = {Szpiro, Lucien and Tepper, Michael and Williams, Phillip},
     TITLE = {Semi-stable reduction implies minimality of the resultant},
   JOURNAL = {J. Algebra},
  FJOURNAL = {Journal of Algebra},
    VOLUME = {397},
      YEAR = {2014},
     PAGES = {489--498},
}

@article {PST09,
    AUTHOR = {Petsche, Clayton and Szpiro, Lucien and Tepper, Michael},
     TITLE = {Isotriviality is equivalent to potential good reduction for
              endomorphisms of {$\Bbb P^N$} over function fields},
   JOURNAL = {J. Algebra},
  FJOURNAL = {Journal of Algebra},
    VOLUME = {322},
      YEAR = {2009},
    NUMBER = {9},
     PAGES = {3345--3365},
}

@article {Lev11,
    AUTHOR = {Levy, Alon},
     TITLE = {The space of morphisms on projective space},
   JOURNAL = {Acta Arith.},
  FJOURNAL = {Acta Arithmetica},
    VOLUME = {146},
      YEAR = {2011},
    NUMBER = {1},
     PAGES = {13--31},
}

@article {Mil93,
    AUTHOR = {Milnor, John},
     TITLE = {Geometry and dynamics of quadratic rational maps},
      NOTE = {With an appendix by the author and Lei Tan},
   JOURNAL = {Experiment. Math.},
  FJOURNAL = {Experimental Mathematics},
    VOLUME = {2},
      YEAR = {1993},
    NUMBER = {1},
     PAGES = {37--83},
}

@article {Sil81,
    AUTHOR = {Silverman, Joseph H.},
     TITLE = {Lower bound for the canonical height on elliptic curves},
   JOURNAL = {Duke Math. J.},
  FJOURNAL = {Duke Mathematical Journal},
    VOLUME = {48},
      YEAR = {1981},
    NUMBER = {3},
     PAGES = {633--648},
}

@misc{Poi24b,
      title={Dynamique analytique sur $\mathbf{Z}$. II : \'Ecart uniforme entre Latt\`es et conjecture de Bogomolov-Fu-Tschinkel}, 
      author={Jérôme Poineau},
      year={2024},
      eprint={2207.01574},
      archivePrefix={arXiv},
      url={https://arxiv.org/abs/2207.01574}, 
}

@misc{Poi24a,
      title={Dynamique analytique sur $\mathbf{Z}$. I : Mesures d'\'equilibre sur une droite projective relative}, 
      author={Jérôme Poineau},
      year={2024},
      eprint={2201.08480},
      archivePrefix={arXiv},
      url={https://arxiv.org/abs/2201.08480}, 
}

@article {Fav20,
    AUTHOR = {Favre, Charles},
     TITLE = {Degeneration of endomorphisms of the complex projective space
              in the hybrid space},
   JOURNAL = {J. Inst. Math. Jussieu},
  FJOURNAL = {Journal of the Institute of Mathematics of Jussieu. JIMJ.
              Journal de l'Institut de Math\'ematiques de Jussieu},
    VOLUME = {19},
      YEAR = {2020},
    NUMBER = {4},
     PAGES = {1141--1183},
}

@article {DF16,
    AUTHOR = {DeMarco, Laura and Faber, Xander},
     TITLE = {Degenerations of complex dynamical systems {II}: analytic and
              algebraic stability},
      NOTE = {With an appendix by Jan Kiwi},
   JOURNAL = {Math. Ann.},
  FJOURNAL = {Mathematische Annalen},
    VOLUME = {365},
      YEAR = {2016},
    NUMBER = {3-4},
     PAGES = {1669--1699},
}

@article {DF14,
    AUTHOR = {De Marco, Laura and Faber, Xander},
     TITLE = {Degenerations of complex dynamical systems},
   JOURNAL = {Forum Math. Sigma},
  FJOURNAL = {Forum of Mathematics. Sigma},
    VOLUME = {2},
      YEAR = {2014},
     PAGES = {Paper No. e6, 36},
}

@article {Kiw14,
    AUTHOR = {Kiwi, Jan},
     TITLE = {Puiseux series dynamics of quadratic rational maps},
   JOURNAL = {Israel J. Math.},
  FJOURNAL = {Israel Journal of Mathematics},
    VOLUME = {201},
      YEAR = {2014},
    NUMBER = {2},
     PAGES = {631--700},
}

@article {Kiw06,
    AUTHOR = {Kiwi, Jan},
     TITLE = {Puiseux series polynomial dynamics and iteration of complex
              cubic polynomials},
   JOURNAL = {Ann. Inst. Fourier (Grenoble)},
  FJOURNAL = {Universit\'e{} de Grenoble. Annales de l'Institut Fourier},
    VOLUME = {56},
      YEAR = {2006},
    NUMBER = {5},
}

@article {Ben07,
    AUTHOR = {Benedetto, Robert L.},
     TITLE = {Preperiodic points of polynomials over global fields},
   JOURNAL = {J. Reine Angew. Math.},
  FJOURNAL = {Journal f\"ur die Reine und Angewandte Mathematik. [Crelle's
              Journal]},
    VOLUME = {608},
      YEAR = {2007},
     PAGES = {123--153},
}

@article {Sil98,
    AUTHOR = {Silverman, Joseph H.},
     TITLE = {The space of rational maps on {$\bold P^1$}},
   JOURNAL = {Duke Math. J.},
  FJOURNAL = {Duke Mathematical Journal},
    VOLUME = {94},
      YEAR = {1998},
    NUMBER = {1},
     PAGES = {41--77},
}

@book {Sil12,
    AUTHOR = {Silverman, Joseph H.},
     TITLE = {Moduli spaces and arithmetic dynamics},
    SERIES = {CRM Monograph Series},
    VOLUME = {30},
 PUBLISHER = {American Mathematical Society, Providence, RI},
      YEAR = {2012},
     PAGES = {viii+140},
}

@article {Luo21,
    AUTHOR = {Luo, Yusheng},
     TITLE = {Limits of rational maps, {$\Bbb{R}$}-trees and barycentric
              extension},
   JOURNAL = {Adv. Math.},
  FJOURNAL = {Advances in Mathematics},
    VOLUME = {393},
      YEAR = {2021},
}

@article {Loo19,
    AUTHOR = {Looper, Nicole},
     TITLE = {A lower bound on the canonical height for polynomials},
   JOURNAL = {Math. Ann.},
  FJOURNAL = {Mathematische Annalen},
    VOLUME = {373},
      YEAR = {2019},
    NUMBER = {3-4},
     PAGES = {1057--1074},
}

@misc{FG24,
      title={Non-Archimedean techniques and dynamical degenerations}, 
      author={Charles Favre and Chen Gong},
      year={2024},
      eprint={2406.15892},
      archivePrefix={arXiv},
}

@book {Ber90,
    AUTHOR = {Berkovich, Vladimir G.},
     TITLE = {Spectral theory and analytic geometry over non-{A}rchimedean
              fields},
    SERIES = {Mathematical Surveys and Monographs},
    VOLUME = {33},
 PUBLISHER = {American Mathematical Society, Providence, RI},
      YEAR = {1990},
     PAGES = {x+169},
}

@article {Bak09,
    AUTHOR = {Baker, Matthew},
     TITLE = {A finiteness theorem for canonical heights attached to
              rational maps over function fields},
   JOURNAL = {J. Reine Angew. Math.},
  FJOURNAL = {Journal f\"ur die Reine und Angewandte Mathematik. [Crelle's
              Journal]},
    VOLUME = {626},
      YEAR = {2009},
     PAGES = {205--233},
}

@misc{Ing22,
      title={Explicit canonical heights for divisors relative to endomorphisms of $\mathbb{P}^N$}, 
      author={Patrick Ingram},
      year={2022},
      eprint={2207.07206},
      archivePrefix={arXiv},
      url={https://arxiv.org/abs/2207.07206}, 
}

@misc{Loo24,
      title={Arakelov-Green's functions for dynamical systems on projective varieties}, 
      author={Nicole R. Looper},
      year={2024},
      eprint={2404.06981},
      archivePrefix={arXiv},
}

@article {Bak06,
    AUTHOR = {Baker, Matthew},
     TITLE = {A lower bound for average values of dynamical {G}reen's
              functions},
   JOURNAL = {Math. Res. Lett.},
  FJOURNAL = {Mathematical Research Letters},
    VOLUME = {13},
      YEAR = {2006},
    NUMBER = {2-3},
     PAGES = {245--257},
}

@article {CS93,
    AUTHOR = {Call, Gregory S. and Silverman, Joseph H.},
     TITLE = {Canonical heights on varieties with morphisms},
   JOURNAL = {Compositio Math.},
  FJOURNAL = {Compositio Mathematica},
    VOLUME = {89},
      YEAR = {1993},
    NUMBER = {2},
     PAGES = {163--205},
}

@book {Ben19,
    AUTHOR = {Benedetto, Robert L.},
     TITLE = {Dynamics in one non-archimedean variable},
    SERIES = {Graduate Studies in Mathematics},
    VOLUME = {198},
 PUBLISHER = {American Mathematical Society, Providence, RI},
      YEAR = {2019},
     PAGES = {xviii+463},
}

@book {Sil07,
    AUTHOR = {Silverman, Joseph H.},
     TITLE = {The arithmetic of dynamical systems},
    SERIES = {Graduate Texts in Mathematics},
    VOLUME = {241},
 PUBLISHER = {Springer, New York},
      YEAR = {2007},
     PAGES = {x+511},
}

@article{DKY20,
    AUTHOR = {DeMarco, Laura and Krieger, Holly and Ye, Hexi},
     TITLE = {Uniform {M}anin-{M}umford for a family of genus 2 curves},
   JOURNAL = {Ann. of Math. (2)},
  FJOURNAL = {Annals of Mathematics. Second Series},
    VOLUME = {191},
      YEAR = {2020},
    NUMBER = {3},
     PAGES = {949--1001},
}

@article{BR06,
    AUTHOR = {Baker, Matthew H. and Rumely, Robert},
     TITLE = {Equidistribution of small points, rational dynamics, and
              potential theory},
   JOURNAL = {Ann. Inst. Fourier (Grenoble)},
  FJOURNAL = {Universit\'{e} de Grenoble. Annales de l'Institut Fourier},
    VOLUME = {56},
      YEAR = {2006},
    NUMBER = {3},
     PAGES = {625--688},
}

@article{Tro16,
    AUTHOR = {Troncoso, Sebastian},
     TITLE = {Bounds for preperiodic points for maps with good reduction},
   JOURNAL = {J. Number Theory},
  FJOURNAL = {Journal of Number Theory},
    VOLUME = {181},
      YEAR = {2017},
     PAGES = {51--72},
}

@article{BR10,
    AUTHOR = {Baker, Matthew and Rumely, Robert},
     TITLE = {Potential theory and dynamics on the {B}erkovich projective
              line},
    JOURNAL = {Mathematical Surveys and Monographs},
    VOLUME = {159},
 PUBLISHER = {American Mathematical Society, Providence, RI},
      YEAR = {2010},
     PAGES = {xxxiv+428},
}

\end{document}